\documentclass{amsproc}
\usepackage{amsfonts}
\usepackage{amsmath}
\usepackage{amssymb}
\usepackage{graphicx}
\usepackage{comment}

\includecomment{comment1}
\excludecomment{newrigid}
\newtheorem{theorem}{Theorem}
\theoremstyle{plain}
\newtheorem{acknowledgement}{Acknowledgement}

\newtheorem{corollary}{Corollary}

\newtheorem{lemma}{Lemma}

\newtheorem{proposition}{Proposition}

\theoremstyle{definition}
\newtheorem{definition}{Definition}
\newtheorem{example}{Example}

\numberwithin{equation}{section}

\begin{document}
\title{Tiling iterated function systems and Anderson-Putnam theory}
\author[M. F. Barnsley]{Michael F. Barnsley}
\address{Australian National University\\
Canberra, ACT, Australia }
\author{Andrew Vince}
\address{University of Florida\\
Gainesville, FL, USA \\
}
\date{11 October 2017}

\begin{abstract}
The theory of fractal tilings of fractal blow-ups is extended to
graph-directed iterated function systems, resulting in generalizations and
extensions of some of the theory of Anderson and Putnam and of Bellisard et
al. regarding self-similar tilings.
\end{abstract}

\maketitle

\section{Introduction \label{sec:intro}}

\begin{figure}[ptb]
\centering\includegraphics[
height=2.0193in,
width=3.832in
]{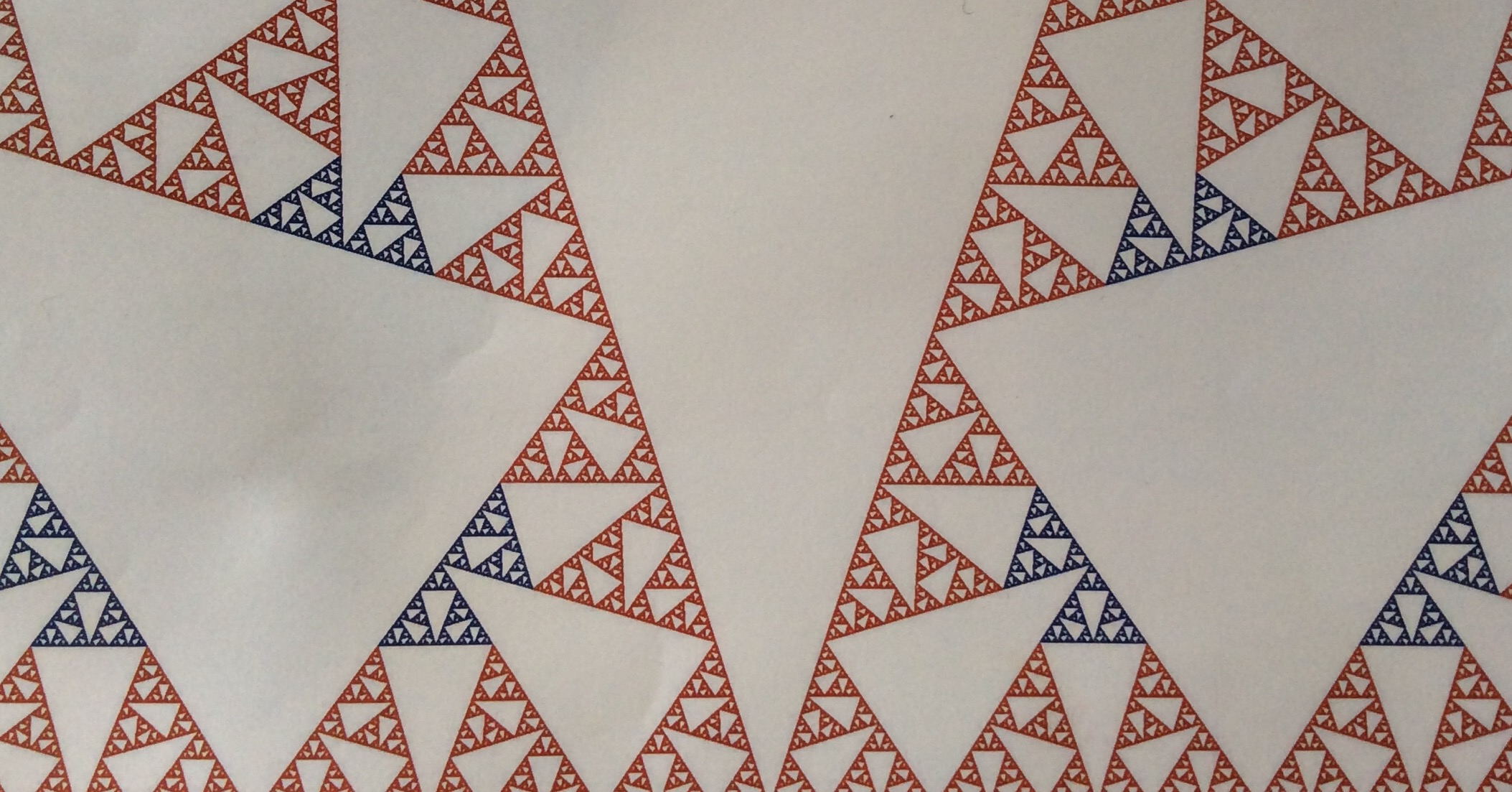}
\caption{Illustration of part of a tiling of a fractal blow-up. There are
two tiles, both topologically conjugate to Sierpinski triangles, a small
blue one and a larger red one.}
\label{img_2010}
\end{figure}

Given a natural number $M$, this paper is concerned with certain tilings of
strict subsets of Euclidean space $\mathbb{R}^{M}$ and of $\mathbb{R}^{M}$
itself. An example of part of such a tiling is illustrated in Figure \ref%
{img_2010}. We substantially generalize the theory of tilings of fractal
blow-ups introduced in \cite{tilings, barnsleyvince} and connect the result
to the standard theory of self-similar tiling \cite{anderson}. The central
main result in this paper is presented in Sections 7 and 8, with
consequences in Section 9. In Section 7 we extend the earlier work by
establishing the exact conditions under which for example translations of a
tiling agree with another tiling; in Section 9 we generalize this result to
tilings of blow-ups of graph directed IFS. Our other main contributions are
(i) development of an algebraic and symbolic fractal tiling theory along the
lines initiated by Bandt \cite{bandt2}; (ii) demonstration that the Smale
system at the heart of Anderson-Putnam theory is conjugate to a type of
symbolic dynamical system that is familiar to researchers in deterministic
fractal geometry; and (iii) to show that the theory applies to tilings of
fractal blow-ups, where the tiles may have no interior and the group of
translations on the tilings space is replaced by a groupoid of isometries.
At the foundational level, this work has notions in common with the work of
Bellisard et al. \cite{bellisard} but we believe that our approach casts new
light and simplicity upon the subject. 
\begin{comment1}

We construct tilings using what we call \textit{tiling iterated function
systems} (TIFS) defined in Section \ref{TIFSsec}. A tiling IFS is a graph
(directed) IFS \cite{jacquin, bedfordGraph, mauldin} where the maps are
similitudes and there is a certain algebraic constraint on the scaling ratios.
When the tilings are recognizable (see \cite{anderson} and references therein)
or more generally the TIFS is \textit{locally rigid}, defined in Section
\ref{rigidsec1}, the tiling space admits an invertible inflation map. Our
construction of a self-similar tiling using a TIFS\ is illustrated in Figure
\ref{15brsc}; it is similar to the one in \cite{tilings, barnsleyvince}, the
key difference being generalization to graph IFS.
\end{comment1}
In the standard theory \cite{anderson} self-similar tilings of $\mathbb{R}%
^{M}$ are constructed by starting from a finite set of CW-complexes which,
after being scaled up by a fixed factor, can be tiled by translations of
members of the original set. (It is easy to see this arrangement may be
described in terms of a graph IFS, that is a finite set of contractive
similitudes that map the set of CW-complexes into itself, together with a
directed graph that describes which maps take which complex into which.) By
careful iteration of this inflation (and subdivision) process, successively
larger patches and, in the limit, tilings may be obtained. We follow a
similar procedure here, but our setting is more general and results in a
rich symbolic understanding of (generalized) tiling spaces.

\begin{figure}[ptb]
\centering
\includegraphics[
%natheight=4.011900in,
%natwidth=5.241600in,
height=4.0119in,
width=5.2416in
]{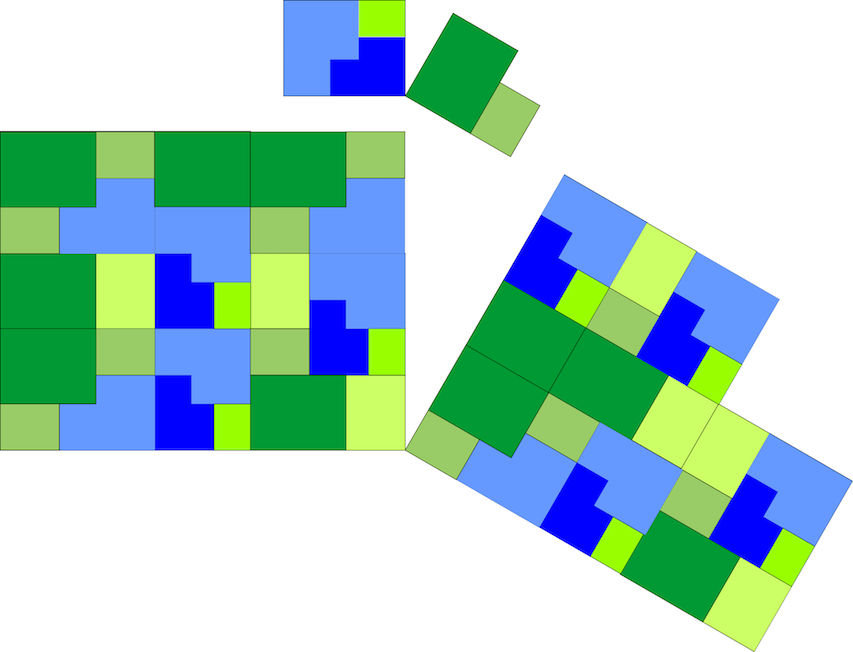}%{l5brsc.png}%
\caption{This figure hints at how a family of tilings may be generated using
a graph directed IFS.}
\label{15brsc}
\end{figure}

In \cite{anderson} it is shown that inflation map acting on a space of
self-similar tilings, as defined there, is conjugate to a shift acting on an
inverse limit space constructed using pointed tilings, and semi-conjugate to
a shift acting on a symbolic space. This raises these questions. When are
two fractal tilings isometric? How can one tell symbolically if two tilings
are isometric? How can one tell if two tiling dynamical systems are
topologically conjugate? What is the topological structure of the tiling
space and how does the tiling dynamical system act on it? For example, can
one see purely symbolically the solenoid structure of the tiling space in
the case of the Penrose tilings, and what happens in the case of purely
fractal tilings? When the tiles are CW-compexes, an approach is via study of
invariants such as zeta functions and cohomology when these can be
calculated, as in \cite{anderson}. Here we approach the answers by
constructing symbolic representation of the (generalized) Anderson-Putnam
complex and associated tiling dynamics.

In Section \ref{sec:one} we provide necessary notation and background
regarding graph IFS and their attractors. We focus on the associated
symbolic spaces, namely the code or address spaces of IFS theory, as these
play a central role. Key to our main results is the relationship between the
attractor $A$ of a graph IFS $\left( \mathcal{F},\mathcal{G}\right) $ and
its address space $\Sigma,$ all defined in Section \ref{sec:one}: this
relationship is captured in a well-known continuous coding map $\pi
:\Sigma\rightarrow A$. Sets of addresses in $\Sigma$ are mapped to points in 
$A$ by $\pi$. This structure is reflected in a second mapping $\Pi
:\Sigma^{\dag}\rightarrow\mathbb{T}$ introduced in Section \ref{tilingsec},
where $\mathbb{T}$ is the tiling space that we associate with $\left( 
\mathcal{F},\mathcal{G}\right) $. The important results of Section \ref%
{sec:one} are the notation introduced there, the information summarized in
Theorem \ref{thm:one}, which concerns the existence and structure of
attractors, and Theorem \ref{thm:two} which concerns the coding map $%
\pi:\Sigma\rightarrow A$.

Section \ref{tilingsec} introduces TIFS (tiling iterated function systems)
and associated tiling space $\mathbb{T}$, shows the existence of a family of
tilings $\{\Pi(\theta):\theta\in\Sigma^{\dag}\}\subset\mathbb{T}$, and
explores their relationship to what we call a canonical family of tilings $%
\{T_{n}^{(v)}\}$ and their symbolic counterparts, symbolic tilings, $%
\{\Omega_{n}^{(v)}\}$, certain subsets of $\Sigma^{\dag}$. We explore the
action of an invertible symbolic inflation operator that acts on the
symbolic tilings and $\{\Omega_{n}^{(v)}\}$ and its inverse. When $%
\Pi:\Sigma^{\dag }\rightarrow\mathbb{T}$ is one-to-one, which occurs when
the TIFS is what we call locally rigid, there is a commutative relationship
between symbolic inflation/deflation on $\{\Omega_{n}^{(v)}\}$ and
inflation/deflation on the range of $\Pi$. **More to go here, then simplify.

In Sections 7 we define relative and absolute addresses of tiles in tilings:
these addressing schemes for tiles in tilings ***. In Section 8 we arrive at
our main result, Theorem **: we characterize members of $\mathbb{T}$ which
are isometric in terms of their addresses. This in turn allows us to
describe the full tiling space, obtained by letting the group of Euclidean
isometries act on the range of $\Pi$. **More to go here.

**Discussion of "forces the borders" (Anderson and Putnam), "unique
composition property" (Solomyak 1997), "recognizability" and relation to
rigid.

\section{\label{sec:one}Attractors of graph directed IFS: notation and
foundational results}

\subsection{Some notation}

$\mathbb{N}$ is the strictly positive integers and $\mathbb{N}_{0}=\mathbb{%
N\cup\{}0\}$. If $\mathcal{S}$ is a finite set, then $\left\vert \mathcal{S}%
\right\vert $ is the number of elements of $\mathcal{S}$ and $\left[ 
\mathcal{S}\right] =\left\{ 1,2,...,\left\vert \mathcal{S}\right\vert
\right\} $. For $N\in\mathbb{N}$, $\left[ N\right] =\{1,2,...,N\},$ $\left[ N%
\right] ^{\ast}=\cup_{k\in\mathbb{N}_{0}}[N]^{k},$ where $[N]^{0}=\left\{
\varnothing\right\} $. Also $d_{N}$ is the metric defined in \cite%
{barnsleyvince} such that $\left( \left[ N\right] ^{\ast }\cup\lbrack N]^{%
\mathbb{N}},d_{N}\right) $ is a compact metric space.

\subsection{Graph directed iterated function systems}

See \cite{hutchinson} for formal background on iterated function systems
(IFS). Here we are concerned with a generalization of IFS, often called
graph IFS. Earlier work related to graph IFS includes \cite{bandtTILE,
barnsleyFE2,bedfordGraph, dekking, mauldin, werner}. In some of these works
graph IFS are referred to as recurrent IFS.

Let $\mathcal{F}$ be a finite set of invertible contraction mappings $f:%
\mathbb{R}^{M}\rightarrow\mathbb{R}^{M}$ with contraction factor $%
0<\lambda<1 $, that is $\left\Vert f(x)-f(y)\right\Vert
\leq\lambda\left\Vert x-y\right\Vert $ for all $x,y\in\mathbb{R}^{M}$. We
write%
\begin{equation*}
\mathcal{F=}\left\{ f_{1},f_{2},...,f_{N}\right\} \text{ with }N=\left\vert 
\mathcal{F}\right\vert \text{.}
\end{equation*}

Let $\mathcal{G}=\left( \mathcal{E},\mathcal{V}\right) $ be a finite
strongly connected directed graph with edges 
\begin{equation*}
\mathcal{E=}\left\{ e_{1},e_{2},...,e_{E}\right\} \text{ with }E=\left\vert 
\mathcal{E}\right\vert =\left\vert \mathcal{F}\right\vert
\end{equation*}
and vertices $\mathcal{V=}\left\{ \upsilon_{1},\upsilon_{2},...,\upsilon
_{V}\right\} $ with $V=\left\vert \mathcal{V}\right\vert \leq\left\vert 
\mathcal{F}\right\vert $.

By "strongly connected" we mean that there is a path, a sequence of
consecutive directed edges, from any vertex to any vertex. There may be
none, one, or more than one directed edges from a vertex to a vertex,
including from a vertex to itself. The set of edges directed from $w\in%
\mathcal{V}$ to $\upsilon\in\mathcal{V}$ in $\mathcal{G}$ is $\mathcal{E}%
_{\upsilon,w}$.

We call $\left( \mathcal{F},\mathcal{G}\right) $ an \textit{graph} \textit{%
IFS }or more fully \textit{a} \textit{graph directed IFS}.\textit{\ }The
graph $\mathcal{G}$ provides the order in which functions of $\mathcal{F}$,
which are associated with the edges, may be composed from left to right. The
sequence directed edges, $(e_{\sigma_{1}},e_{\sigma_{2}},...e_{\sigma_{k}})$%
, is associated with the composite function $f_{\sigma_{1}}\circ
f_{\sigma_{2}}\circ...\circ f_{\sigma_{k}}$. We may denote the edges by
their indices $\left\{ 1,2,...,E\right\} $ and the vertices by $\left\{
1,2,...,V\right\} $.

\begin{figure}[ptb]
\centering
\includegraphics[
height=3.2024in,
width=5.2278in
]{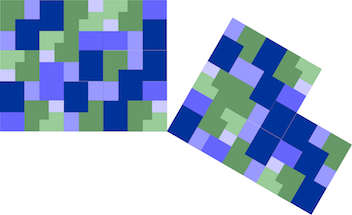}
\caption{See text.}
\label{blevel7col}
\end{figure}

Reference to Figure \ref{blevel7col}.

\subsection{Addresses of directed paths}

Let $\Sigma_{k}$ be the set of directed paths in $\mathcal{G}$ of length $%
k\in\mathbb{N}$, let $\Sigma_{0}$ be the empty string $\varnothing$, and $%
\Sigma_{\infty}$ be the set of directed paths, each of which starts at a
vertex and is of infinite length. Define 
\begin{equation*}
\Sigma=\Sigma_{\ast}\cup\Sigma_{\infty}\text{ where }\Sigma_{\ast}:=\cup
_{k\in\mathbb{N}_{0}}\Sigma_{k}\text{. }
\end{equation*}
A point or path $\sigma\in\Sigma_{k}$ is represented by $\sigma=\sigma
_{1}\sigma_{2}...\sigma_{k}\in\left[ N\right] ^{k}$ corresponding to the
sequence of edges $(e_{\sigma_{1}},e_{\sigma_{2}},...e_{\sigma_{k}})$
successively encountered on a directed path of length $k$ in $\mathcal{G}$.
Paths in $\Sigma$ correspond to allowed compositions of functions of the IFS.

Let $\mathcal{G}^{\dag}=(\mathcal{E}^{\dag},\mathcal{V})$ be the graph $%
\mathcal{G}$ modified so that the directions of all edges are reversed. Let $%
\Sigma_{k}^{\dag}$ be the set of directed paths in $\mathcal{G}^{\dag}$ of
length $k$. Let ${\Sigma}_{\infty}^{\dag}$ be the set of directed paths of $%
\mathcal{G}^{\dag}$, each of which start at a vertex and is of infinite
length. Let 
\begin{equation*}
{\Sigma}^{\dag}={\Sigma}_{\ast}^{\dag}\cup{\Sigma}_{\infty}^{\dag}\text{
where }{\Sigma}_{\ast}^{\dag}:=\cup_{k\in\mathbb{N}_{0}}{\Sigma}_{k}^{\dag}%
\text{. }
\end{equation*}

We define 
\begin{equation*}
\mathcal{E}_{v,\ast}=\left\{ u\in\mathcal{E}_{\upsilon,w}:w\in\mathcal{V}%
\right\} .
\end{equation*}
and, in the obvious way, also define $\mathcal{E}_{\ast,v}=$ $\mathcal{E}%
_{v,\ast}^{\dag}$, $\mathcal{E}_{w,v}=$ $\mathcal{E}_{v,w}^{\dag}$.

\begin{figure}[ptb]
\centering
\includegraphics[
height=3.2085in,
width=5.2382in
]{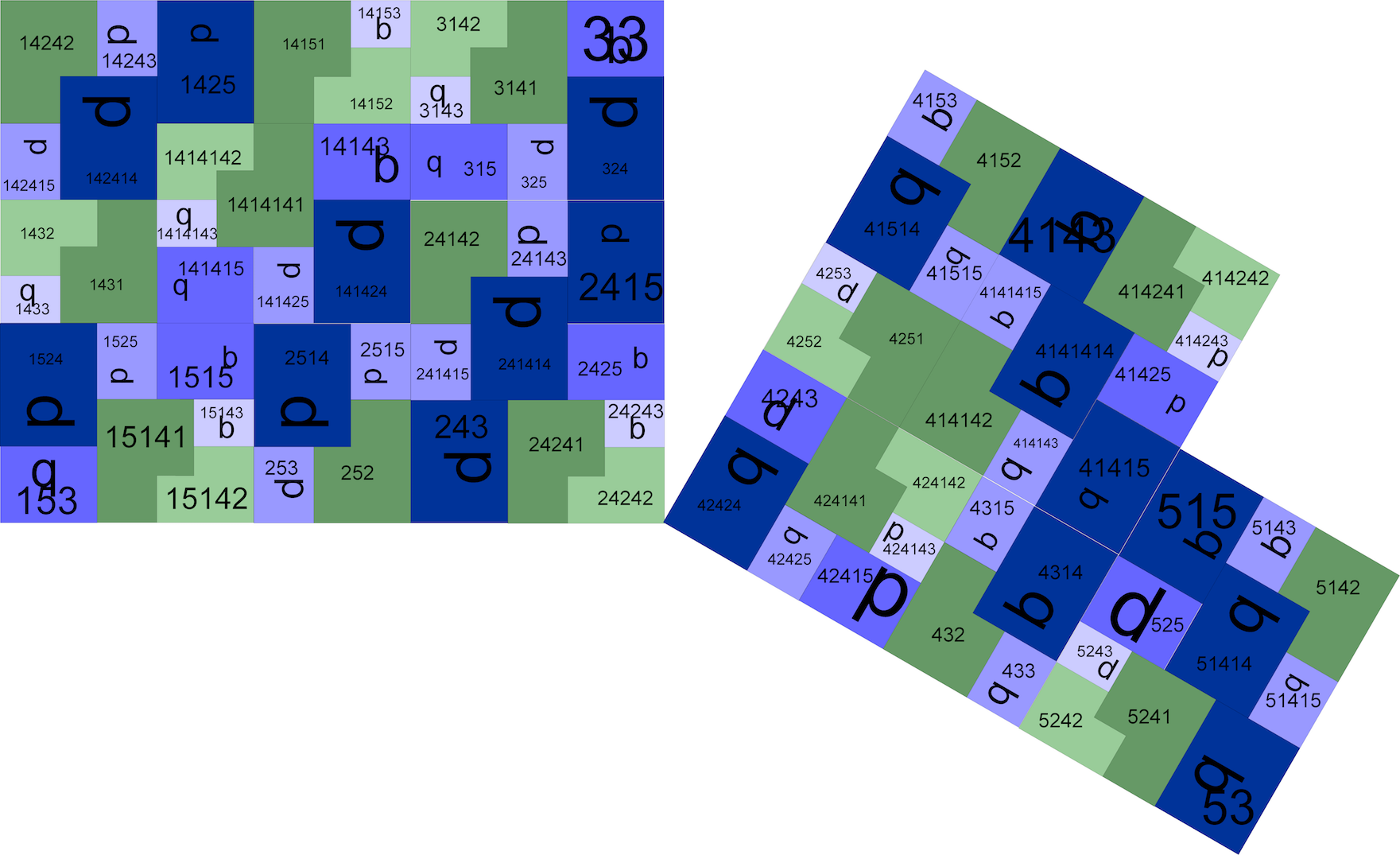}
\caption{See text.}
\label{blevel7colbs}
\end{figure}

\subsection{Notation for compositions of functions}

For $\theta=\theta_{1}\theta_{2}\cdots\theta_{k}\in\lbrack N]^{\ast}$, the
following notation will be used: 
\begin{equation*}
\begin{aligned} f_{\theta} &= f_{\theta_{1}} f_{\theta_{2}}\cdots
f_{\theta_k} \\ f_{-\theta} &=f_{\theta_{1}}^{-1}f_{\theta_{2}}^{-1}\cdots
f_{\theta_k}^{-1}=(f_{\theta_k\theta_{k-1}\cdots\theta_{1}})^{-1},
\end{aligned}
\end{equation*}
with the convention that $f_{\theta}$ and $f_{-\theta}$ are the identity
function $Id_{\mathbb{R}^{M}}$ if $\theta=\varnothing$. Likewise, for all $%
\theta\in\lbrack N]^{\mathbb{N}}$ and $k\in\mathbb{N}_{0},$ define $%
\theta|k=\theta_{1}\theta_{2}\cdots\theta_{k}$ and%
\begin{equation*}
f_{-\theta|k}=f_{\theta_{1}}^{-1}f_{\theta_{2}}^{-1}\cdots
f_{\theta_{k}}^{-1}=(f_{\theta_{k}\theta_{k-1}\cdots\theta_{1}})^{-1},
\end{equation*}
with the convention that $f_{-\theta|0}=id$.

For $\theta\in\lbrack N]^{\ast}\cup\lbrack N]^{\mathbb{N}}$ we define 
\begin{equation*}
\left\vert \theta\right\vert =\left\{ 
\begin{array}{cc}
k & \text{if }\theta\in\lbrack N]^{k}, \\ 
\infty & \text{if }\theta\in\lbrack N]^{\mathbb{N}}.%
\end{array}
\right.
\end{equation*}

\subsection{Existence and approximation of attractors}

Let $\mathbb{H}$ be the nonempty compact subsets of $\mathbb{R}^{M}$. We
equip $\mathbb{H}$ with the Hausdorff metric so that it is a complete metric
space. Define $\mathbf{F:}\mathbb{H}^{V}\rightarrow\mathbb{H}^{V}$ by%
\begin{equation*}
\left( \mathbf{F}X\right) _{v}=\left\{ x\in f_{e}X_{w}:e\in\mathcal{E}%
_{v,w},w\in\mathcal{V}\right\} \text{,}
\end{equation*}
for all $X\in\mathbb{H}^{V}$, where $X_{w}$ is the $w^{th}$ component of $X$.

\begin{definition}
Define $\theta\in\Sigma_{\infty}$ to be \textbf{disjunctive} if, given any $%
k\in\mathbb{N}$ and $\omega\in\Sigma_{k}$ there is $p\in\mathbb{N}_{0}$ so
that $\omega=\theta_{p+1}\theta_{p+1}...\theta_{p+k}$.
\end{definition}

Theorem \ref{thm:one} summarizes some known or readily inferred information
regarding the existence, uniqueness, and construction of attractors of $(%
\mathcal{F},\mathcal{G)}$.

\begin{theorem}
\label{thm:one} Let $\left( \mathcal{F},\mathcal{G}\right) $ be a graph
directed IFS.

\begin{enumerate}
\item (\textbf{Contraction on} $\mathbb{H}^{V}$) The map $\mathbf{F}:\mathbb{%
H}^{V}\rightarrow\mathbb{H}^{V}$ is a contraction with contractivity factor $%
\lambda$. There exists unique $\mathbf{A}=(A^{1},A^{2},...,A^{V})\in\mathbb{H%
}^{V}$ such that 
\begin{equation*}
\mathbf{A=FA}
\end{equation*}
and 
\begin{equation*}
\mathbf{A=}\lim_{k\rightarrow\infty}\mathbf{F}^{k}\mathbf{B}
\end{equation*}
for all $\mathbf{B}\in\mathbb{H}^{V}$.

\item (\textbf{Chaos Game on }$\mathbb{H}$) There is a unique $A\in\mathbb{H}$
such that
\[
A=\bigcap\limits_{k\in\mathbb{N}}\overline{(%
%TCIMACRO{\dbigcup \limits_{n=k}^{\infty}}%
%BeginExpansion
{\displaystyle\bigcup\limits_{n=k}^{\infty}}
%EndExpansion
x_{n})},
\]
for all $x_{0}\in\mathbb{R}^{m},$ and all disjunctive $\theta=\theta_{1}%
\theta_{2}...\in\Sigma_{\infty}^{\dag}$. Here
\[
x_{n}=f_{\theta_{n}}(x_{n-1})
\]
for all $n\in\mathbb{N}$ and the bar denotes closure. The set $A$ is related
to $\mathbf{A}$ by $A=\cup_{v\in\mathcal{V}}A^{v}.$

\item (\textbf{Deterministic Algorithm on }$\mathbb{H}$) If $B\in\mathbb{H}$
then 
\begin{equation*}
A=\lim_{k\rightarrow\infty}\left\{ x\in
f_{\sigma}(B):\sigma=\sigma_{1}\sigma_{2}...\sigma_{k}\in\Sigma_{k}\right\} .
\end{equation*}
Also 
\begin{equation*}
A^{w}=\lim_{k\rightarrow\infty}\left\{ x\in f_{\sigma}(B):\sigma=\sigma
_{1}\sigma_{2}...\sigma_{k}\in\Sigma_{k},\sigma_{1}\in\mathcal{E}_{\ast
,w}\right\}
\end{equation*}
for all $w\in\mathcal{V}$.
\end{enumerate}
\end{theorem}

\begin{proof}
(1) The proof of this is well-known and straightfoward. See for example \cite%
[Chapter 10]{barnsleyFE2}.

(2) This is a simple generalization of the main result in \cite%
{barnsleychaos} which applies when $\left\vert \mathcal{V}\right\vert =1$.

(3)\ This follows from (1).
\end{proof}

\begin{definition}
Using the notation of Theorem \ref{thm:one}, $A:=\cup_{v\in\mathcal{V}}A^{v}$
is the \textbf{attractor} of the IFS $\mathcal{(F},\mathcal{G)}$ and $%
\left\{ A^{v}:v\in\mathcal{V}\right\} $ are its components.
\end{definition}

We adopt this definition because it is unified with the case $\left\vert 
\mathcal{V}\right\vert =1$, allowing us to work using only of one copy of $%
\mathbb{R}^{M}$ and to provide a tiling theory that is naturally unified to
all cases. See also \cite{bandt2}. Algorithms based on the chaos game that
plot and render pictures of attractors in $\mathbb{R}^{M}$ when $\left\vert 
\mathcal{V}\right\vert =1$ can be generalized by restricting the symbolic
orbits so that they are consistent with the graph.

In this paper we assume $A^{i}\cap A^{j}=\varnothing$ for $i\neq j$. When
this is not the case, components of the attractor can be moved around to
ensure that they have empty intersections by means of a simple change of
coordinates: the replacements $f_{i}$ $\rightarrow$ $T_{v}f_{i}T_{v}^{-1}$
for all $i\in\mathcal{E}_{v,v}$, $f_{i}$ $\rightarrow$ $T_{v}f_{i}$ for all $%
i\in\mathcal{E}_{v,\ast}^{\dag}\backslash\mathcal{E}_{v,v}$, $f_{i}$ $%
\rightarrow$ $f_{i}T_{\upsilon}^{-1}$ for all $i\in\mathcal{E}%
_{\ast,v}^{\dag}\backslash\mathcal{E}_{v,v}$, where $T_{\upsilon}:\mathbb{R}%
^{M}\rightarrow\mathbb{R}^{M}$ is for example a translation, moves $A^{v}$
to $T_{\upsilon}A^{v}$ without altering the other components of the
attractor.

\subsection{The coding map $\protect\pi:\Sigma\rightarrow\mathbb{H(}A)$}

For $e\in\mathcal{E}$, let $\overleftarrow{\upsilon}(e),\overrightarrow{%
\upsilon}(e)\in\mathcal{V}$ be the unique vertices such that $e$ is directed
from $\overleftarrow{\upsilon}(e)$ to $\overrightarrow{\upsilon}(e)$.

\begin{definition}
Define ${\pi}:\Sigma\rightarrow\mathbb{H(}A)$ by%
\begin{align*}
\mathcal{\pi}(\varnothing) & =A, \\
\mathcal{\pi}(\omega) & =f_{\omega}(A^{\overrightarrow{\upsilon}(\omega
_{k})})\text{ for all }\omega=\omega_{1}\omega_{2}...\omega_{k}\in\Sigma
_{\ast},k\in\mathbb{N} \\
\mathcal{\pi}(\sigma) & =\lim_{k\rightarrow\infty}\pi(\omega|k)\text{,}\ 
\text{for all }\sigma\in\Sigma_{\infty},
\end{align*}
where the limit is with respect to the Hausdorff metric on $\mathbb{H(}A),$
the collection of nonempty compact subsets of $A$.
\end{definition}

We call $\Sigma$ the \textbf{address space }or\textbf{\ code space} and $\pi$
the \textbf{coding map} for the attractor of the graph IFS $\mathcal{(F},%
\mathcal{G)}$.

\begin{theorem}
\label{thm:two}The map $\pi:\Sigma\rightarrow\mathbb{H(}A)$ is well-defined
and continuous. Restricted to $\Sigma_{\infty}$, $\pi$ is a continuous map
from $\Sigma_{\infty}$ into $\mathbb{R}^{M}$ and $\pi(\Sigma_{\infty})=\{%
\pi(\sigma):\sigma\in\Sigma_{\infty}\}=A$.
\end{theorem}

\begin{proof}
This follows the same lines as for the case $\left\vert \mathcal{V}%
\right\vert =1$ and is well known since the work of Hutchinson \cite%
{hutchinson}.
\end{proof}

\subsection{Shift maps}

Shift maps acting on the symbolic spaces $\Sigma$ and ${\Sigma}^{\dag}$,
defined here, will be seen to interact in an important way with coding maps,
attractors, and tilings.

The \textit{shift }$S:[N]^{\ast}\cup\lbrack N]^{\infty}\rightarrow\lbrack
N]^{\ast}\cup\lbrack N]^{\infty}$ is defined by $S(\theta_{1}\theta_{2}%
\cdots\theta_{k})=\theta_{2}\theta_{3}\cdots\theta_{k}$ and $S(\theta
_{1}\theta_{2}\cdots)=\theta_{2}\theta_{3}\cdots,$ with the convention that $%
S\theta_{1}=\varnothing$. A point $\theta\in\lbrack N]^{\infty}$ is \textit{%
eventually periodic} if there exists $m\in\mathbb{N}_{0}$ and $n\in\mathbb{N}
$ such that $S^{m}\theta=S^{m+n}\theta$. In this case we write $%
\theta=\theta_{1}\theta_{2}\cdots\theta_{m}\overline{\theta_{m+1}\theta
_{m+2}\cdots\theta_{m+n}}$.

We have $S(\Sigma)=\Sigma$ and $S({\Sigma}^{\dag})={\Sigma}^{\dag}.$ We
write $S:\Sigma\rightarrow\Sigma$ for the restricted map $S|_{\Sigma}$ and
likewise write $S:{\Sigma}^{\dag}\rightarrow{\Sigma}^{\dag}$. Note that $S$
is continuous. The metric spaces $\left( \Sigma,d_{\left\vert \mathcal{F}%
\right\vert }\right) $ and $\left( {\Sigma}^{\dag},d_{\left\vert \mathcal{F}%
\right\vert }\right) $ are compact shift invariant subspaces of $\left[ N%
\right] ^{\ast}\cup\left[ N\right] ^{\infty}$.

The coding map $\pi:\Sigma\rightarrow\mathbb{H(}A)$ interacts with shift $%
S:\Sigma\rightarrow\Sigma$ according to%
\begin{equation*}
f_{\sigma|k}\circ\pi\circ S^{k}\left( \sigma\right) =\pi\left( \sigma\right)
\end{equation*}
for all $\sigma\in\Sigma$, for all $k\in\mathbb{N}$ with $k\leq\left\vert
\sigma\right\vert $.

\section{\label{tilingsec}Tilings}

\subsection{Tilings in this paper}

We use the same definitions of tile, tiling, similitude, scaling ratio,
isometry and prototile set as in \cite{barnsleyvince}. For completeness we
quote the definitions in this section. A \textit{tile} is a perfect (i.e. no
isolated points) compact nonempty subset of $\mathbb{R}^{M}$. Fix a
Hausdorff dimension $0<D_{H}\leq M$. A \textit{tiling} in $\mathbb{R}^{M}$
is a set of tiles, each of Hausdorff dimension $D_{H}$, such that every
distinct pair of tiles is non-overlapping. Two tiles are \textit{%
non-overlapping} if their intersection is of Hausdorff dimension strictly
less than $D_{H}$. The \textit{support} of a tiling is the union of its
tiles. We say that a tiling tiles its support.

A \textit{similitude} is an affine transformation $f:{\mathbb{R}}%
^{M}\rightarrow {\mathbb{R}}^{M}$ of the form $f(x)=\lambda \,O(x)+q$, where 
$O$ is an orthogonal transformation and $q\in \mathbb{R}^{M}$ is the
translational part of $f(x)$. The real number $\lambda >0$, a measure of the
expansion or contraction of the similitude, is called its \textit{scaling} 
\textit{ratio}. An \textit{isometry} is a similitude of unit scaling ratio
and we say that two sets are isometric if they are related by an isometry.
We write $\mathcal{U}$ to denote the group of isometries on $\mathbb{R}^{M}$
and write $\mathcal{T}$ to denote a specific group contained in $\mathcal{U}$%
, see above Lemma \ref{lem:canonical}.

The \textit{prototile set} $\mathcal{P}$ of a tiling $T$ is a set of tiles
such that every tile $t\in T$ can be written in the form $\tau(p)$ for some $%
\tau\in\mathcal{T}$ and $p\in\mathcal{P}$. The tilings constructed in this
paper have finite prototile sets.

\subsection{A convenient compact tiling space}

Let $\mathbb{T}^{^{\prime}}$ be the set of all tilings on $\mathbb{R}^{M}$
using a fixed prototile set (and fixed group $\mathcal{T}$). Let $%
t_{\emptyset}$ be the empty tile of $\mathbb{R}^{M}$. We assume throughout
that if $T\in\mathbb{T}^{^{\prime}}$ then $t_{\emptyset}\in T$. We may think
of $t_{\emptyset}$ as "the tile at infinity".

Let $\rho:\mathbb{R}^{M}\rightarrow\mathbb{S}^{M}$ be the usual $M$%
-dimensional stereographic projection to the $M$-sphere, obtained by
positioning $\mathbb{S}^{M}$ tangent to $\mathbb{R}^{M}$ at the origin.
Define $\widehat{\rho}:\mathbb{T}^{^{\prime}}\rightarrow\mathbb{S}^{M}$ so
that $\widehat{\rho}\left( t_{\emptyset}\right) =\mathbb{S}^{M}\backslash
\rho(\mathbb{R}^{M})$ is the point on $\mathbb{S}^{M}$ diametric to the
origin and%
\begin{equation*}
\widehat{\rho}\left( T\right) =\left\{ \rho\left( t\right) :t\in T,t\neq
t_{0}\right\} \cup\widehat{\rho}\left( t_{\emptyset}\right) \text{.}
\end{equation*}
Let $\mathbb{H(S}^{M})$ be the non-empty closed (w.r.t. the usual topology
on $\mathbb{S}^{M}$) subsets of $\mathbb{S}^{M}.$ Let $d_{\mathbb{H(S}^{M})}$
be the Hausdorff distance with respect to the round metric on $\mathbb{S}%
^{M} $, so that $(\mathbb{H(S}^{M}),d_{\mathbb{H(S}^{M})})$ is a compact
metric space. Let $\mathbb{H(H(S}^{M}))$ be the nonempty compact subsets of $%
(\mathbb{H(S}^{M}),d_{\mathbb{H(S}^{M})})$, and let $d_{\mathbb{H(H(S}%
^{M}))} $ be the associated Hausdorff metric. Then $(\mathbb{H(H(S}^{M})),d_{%
\mathbb{H(H(S}^{M}))})$ is a compact metric space. Finally, define a metric $%
d_{\mathbb{T}^{\prime}}$ on $\mathbb{T}^{\prime}$ by%
\begin{equation*}
d_{\mathbb{T}^{\prime}}(T_{1},T_{2})=d_{\mathbb{H(H(S}^{M}))}(\widehat{\rho }%
\left( T_{1}\right) ,\widehat{\rho}\left( T_{2}\right) )
\end{equation*}
for all $T_{1},T_{2}\in\mathbb{T}^{\prime}$.

\begin{theorem}
$\mathbb{(T}^{\prime},d_{\mathbb{T}^{\prime}})$ is a compact metric space.
\end{theorem}

\begin{proof}
Note that a sequence of tilings in $\mathbb{T}^{\prime}$ may converge to $%
t_{\emptyset}$, but this cannot happen if all the tilings in a sequence have
a nonempty tile in common.
\end{proof}

See also \cite{anderson, bedfordScene, sadun, solomyak, wicks} where other,
mainly equivalent, metrics and topologies on various tiling spaces are
defined and discussed.

\subsection{\label{TIFSsec}Tiling iterated function systems}

\begin{definition}
\label{defONE}Let $\mathcal{F}=\{{\mathbb{R}}^{M};f_{1},f_{2},\cdots,f_{N}\}$%
, with $N\geq2$, be an IFS of contractive similitudes where the scaling
factor of $f_{n}$ is $\lambda_{n}=s^{a_{n}}$ where $a_{n}\in\mathbb{N}$ and $%
\gcd\{a_{1},a_{2},\cdots,a_{N}\}=1$ and we define%
\begin{equation*}
a_{\max}=\max\{a_{i}:i=1,2,\dots,N\}.
\end{equation*}
For $x\in{\mathbb{R}}^{M}$, the function $f_{n}:\mathbb{R}^{M}\rightarrow 
\mathbb{R}^{M}$ is defined by 
\begin{equation*}
f_{n}(x)=s^{a_{n}}O_{n}(x)+q_{n}
\end{equation*}
where $O_{n}$ is an orthogonal linear transformation and $q_{n}\in{\mathbb{R}%
}^{M}$. Let $D_{H}(X)$ be the Hausdorff dimension of $X\subset\mathbb{R}^{M}$%
. We require that the graph $\mathcal{G}$ be such that 
\begin{equation}
D_{H}(f_{e}(A^{\overrightarrow{v}(e)})\cap f_{l}(A^{\overrightarrow{v}%
(l)}))<D_{H}(A)  \label{equation0}
\end{equation}
for all $e,l\in\mathcal{E}$ with $e\neq l$. We also require%
\begin{equation}
A^{i}\cap A^{j}=\varnothing  \label{equation01}
\end{equation}
for all $i\neq j$. If these conditions and the requirement on $\mathcal{F}$
above hold, then we say that $\left( \mathcal{F},\mathcal{G}\right) $ is a 
\textbf{tiling iterated function system} or \textbf{TIFS}\textit{.}
\end{definition}

It might be better to require that $(\mathcal{F}$,$\mathcal{G)}$ obeys the
open set condition (OSC) namely, there exists a nonempty bounded open set $U$
so that $f_{i}(U)\subset U$ and $f_{i}(U)\cap f_{j}(U)=\varnothing$ for all $%
i\neq j$, $i,j\in\lbrack N]$.

Note that when $(\mathcal{F}$,$\mathcal{G)}$ obeys the OSC $D_{H}$ can be
described elegantly using a spectral radius, see \cite{mauldin}, as follows.
For given $e,l\in\mathcal{E}$, define%
\begin{equation*}
\mathcal{V}_{e,l}=\left\{ 
\begin{array}{cc}
1\text{ } & \text{if }e\text{ follows }l\text{, } \\ 
0\text{ } & \text{otherwise.}%
\end{array}
\right.
\end{equation*}
If $(\mathcal{F}$,$\mathcal{G)}$ obeys the OSC then $D_{H}\in(0,M]$ is the
unique value such that the spectral radius of the $N\times N$ matrix $%
\mathcal{V}_{i,j}s^{D_{H}a_{j}}$ equals one. In this case we expect, based
on what happens in the case of standard IFS theory, discussed a bit in \cite%
{barnsleyvince}, that Equation \ref{equation0} holds, and our theory
applies. In the case $\left\vert \mathcal{V}\right\vert =1$ the OSC implies
that the Hausdorff dimension of $A$ is strictly greater than the Hausdorff
dimension of the \textit{set of overlap} $\mathcal{O=\cup}_{i\neq
j}f_{i}(A)\cap f_{j}(A)$. Similitudes applied to subsets of the set of
overlap comprise the sets of points at which tiles may meet. In \cite[p.481]%
{bandt} we discuss measures of attractors compared to measures of the set of
overlap.

\subsection{The function $\protect\xi:\Sigma_{\ast}\rightarrow\mathbb{N}_{0}$
and the addresses $\Omega_{k}$}

For $\sigma=\sigma_{1}\sigma_{2}\cdots\sigma_{k}\in\Sigma_{\ast}$ define 
\begin{equation*}
\xi(\sigma)=a_{\sigma_{1}}+a_{\sigma_{2}}+\cdots+a_{\sigma_{k}}\qquad \text{%
and}\qquad\xi^{-}(\sigma)=a_{\sigma_{1}}+a_{\sigma_{2}}+\cdots
+a_{\sigma_{k-1}},
\end{equation*}
and $\xi(\varnothing)=\xi^{-}(\varnothing)=0$. We also write $\sigma
^{-}=\sigma_{1}\sigma_{2}\cdots\sigma_{k-1}$ so that 
\begin{equation*}
\xi^{-}(\sigma)=\xi(\sigma^{-}).
\end{equation*}
Define 
\begin{align*}
\Omega_{k} & =\{\sigma\in\Sigma_{\ast}:\xi^{-}(\sigma)\leq k<\xi (\sigma)\},
\\
\Omega_{0} & =\left[ N\right] ,
\end{align*}
for all $k\in\mathbb{N}$ and $\upsilon\in\mathcal{V}$.

\subsection{Our tilings $\left\{ \Pi\left( \protect\theta\right) :\protect%
\theta\in \Sigma^{\dag}\right\} $}

\begin{definition}
A mapping $\Pi$ from ${\Sigma}^{\dag}$ to collections of subsets of $\mathbb{%
H(R}^{M})$ is defined as follows. For $\theta\in\Sigma_{\ast}^{\dag
},\theta\neq\left\{ \emptyset\right\} ,$%
\begin{equation*}
\Pi(\theta_{1}\theta_{2}...\theta_{k}):=\{f_{_{-\theta_{1}\theta_{2}...%
\theta_{k}}}\pi\left( \sigma\right) :\sigma\in\Omega_{\xi(\theta
_{1}\theta_{2}...\theta_{k})},\sigma_{1}\in\mathcal{E}_{\overrightarrow{%
\upsilon}(\theta_{k}),\ast}\},
\end{equation*}
and for $\theta\in\Sigma_{\infty}^{\dag}$ 
\begin{equation*}
\Pi(\theta):=\bigcup\limits_{k\in\mathbb{N}}\Pi(\theta|k).
\end{equation*}
Also 
\begin{equation*}
\mathbb{T}:=\Pi\left( \Sigma^{\dag}\right) ,\mathbb{T}_{\infty}:=\Pi\left(
\Sigma_{\infty}^{\dag}\right) ,\mathbb{T}_{\ast}:=\Pi\left( \Sigma_{\ast
}^{\dag}\right)
\end{equation*}
\end{definition}

\begin{definition}
We say that $(\mathcal{F},\mathcal{G)}$ is purely self-referential if $%
\mathcal{E}_{v,v}\neq\varnothing$ for all $v\in\mathcal{V}$.
\end{definition}

**If $(\mathcal{F},\mathcal{G)}$ is such that $\mathcal{E}%
_{v,v}\neq\varnothing$ for at least one $v\in\mathcal{V}$, then by composing
functions along paths through vertices for which $\mathcal{E}%
_{v,v}=\varnothing$, assigning indices to these composed functions, and
relabelling the functions and redefining the tiles, we can obtain a
self-referential TIFS which presents the essential action of the system.

\begin{theorem}
\label{thm:three} Let $(\mathcal{F},\mathcal{G)}$ be a TIFS.

\begin{enumerate}
\item Each set $\Pi(\theta)$ in $\mathbb{T}$ is a tiling of a subset of $%
\mathbb{R}^{M}$, the subset being bounded when $\theta\in{\Sigma}_{\ast
}^{\dag}$ and unbounded when $\theta\in{\Sigma}_{\infty}^{\dag}$.

\item For all $\theta\in{\Sigma}_{\infty}^{\dag}$ the sequence of tilings $%
\left\{ \Pi(\theta|k)\right\} _{k=1}^{\infty}$ is nested according to 
\begin{equation}
\Pi(\theta|1)\subset\Pi(\theta|2)\subset\Pi(\theta|3)\subset\cdots\text{ .}
\label{eqthmONE}
\end{equation}

\item If $(\mathcal{F},\mathcal{G)}$ is purely self-referential, then for
all $\theta\in{\Sigma}^{\dag}$, with $\left\vert \theta\right\vert $
sufficiently large, the prototile set for $\Pi(\theta)$ is 
\begin{equation*}
\mathcal{P}:=\{s^{i}A^{v}:i\in\left\{ 1,2,\cdots,a_{\max}\right\} ,v\in%
\mathcal{V}\}.
\end{equation*}

\item \textit{Let }$\mathbb{T}^{\prime}$ to be the set of all tilings with
prototile set $\mathcal{P}$. \textit{The map}%
\begin{equation*}
\Pi:{\Sigma}^{\dag}\rightarrow\mathbb{T\subset T}^{\prime}
\end{equation*}
is continuous from the compact metric space $\left( {\Sigma}^{\dag
},d_{\left\vert \mathcal{F}\right\vert }\right) $ into the compact metric
space $(\mathbb{T}^{\prime},d_{\mathbb{T}^{\prime}})$.

\item For all $\theta\in{\Sigma}_{\infty}^{\dag}$ 
\begin{equation}
\Pi(\theta)=\lim_{k\rightarrow\infty}f_{_{-\theta|k}}(\{\pi\left(
\sigma\right) :\sigma\in\Omega_{\xi(\theta|k)}\}).  \label{eqn:whole}
\end{equation}
\end{enumerate}
\end{theorem}

\begin{proof}
Concerning (5) : Since the components of the attractor are "just touching"
or have empty intersection, we have Equation (\ref{eqn:whole}). The $k^{th}$
term here is the function $f_{_{-\theta|k}}$ applied to the whole attractor
(the union of all of the $A^{v}$) "refined or subdivided" to depth $k$. To
say this another way: the set inside the curly parentheses is the \textbf{%
whole} attractor, the union of its components, partitioned systematically
recursively $k$ times.
\end{proof}

\section{\label{symbolicsec}Symbolic structure:\ canonical symbolic tilings
and symbolic inflation and deflation}

In this section we develop notation and key results concerning what we might
call symbolic tiling theory. In Section \ref{localsec}, we show that these
symbolic structures and relationships are conjugate to counterparts in
self-similar tiling theory. These concepts are also interesting because of
their combinatorial structure.

Define 
\begin{equation*}
\Sigma_{\ast}^{v}=\{\sigma\in\Sigma_{\ast}:\sigma_{1}\in\mathcal{E}_{\ast
,v}\},\Sigma_{\infty}^{v}=\{\sigma\in\Sigma_{\infty}:\sigma_{1}\in \mathcal{E%
}_{\ast,v}\},\Sigma^{v}=\{\sigma\in\Sigma:\sigma_{1}\in \mathcal{E}%
_{\ast,v}\}
\end{equation*}
for all $v\in\mathcal{V}$, and analogously define $\Sigma_{\ast}^{\dag
v},\Sigma_{\infty}^{\dag v},\Sigma^{\dag v}.$ Define what we might call
canonical symbolic tilings 
\begin{equation*}
\Omega_{k}^{v}=\{\sigma\in\Sigma_{\ast}^{v}:\xi^{-}(\sigma)\leq k<\xi
(\sigma)\},
\end{equation*}
for all $k\in\mathbb{N}$ and $\upsilon\in\mathcal{V}$. Note that 
\begin{equation*}
\Omega_{k}=\cup_{v\in\mathcal{V}}\Omega_{k}^{v}\text{ and }%
\Omega_{0}^{v}=\{j\in\lbrack N]:\sigma_{1}\in\mathcal{E}_{\ast,v}\}
\end{equation*}

We write $\Omega_{k}^{(v)}$ to mean any one of the sets $\Omega_{k}$ and $%
\Omega_{k}^{v}$ for $v\in\mathcal{V}$. The following lemma tells us that $%
\Omega_{k+1}^{(v)}$ can be obtained from $\Omega_{k}^{(v)}$ by adding
symbols to the right-hand end of some strings in $\Omega_{k}^{(v)}$ and
leaving the other strings unaltered.

\begin{lemma}
\label{lemma:split}(\textbf{Symbolic Splitting}) For all $k\in\mathbb{N}$
and $v\in\mathcal{V}$ the following relations hold:%
\begin{equation*}
\Omega_{k+1}^{(v)}=\left\{ \sigma\in\Omega_{k}^{(v)}:k+1\leq\xi\left(
\sigma\right) \right\} \cup\left\{ \sigma j\in\Sigma_{\ast}^{\left( v\right)
}:\sigma\in\Omega_{k}^{(v)},k=\xi\left( \sigma\right) \right\} \text{.}
\end{equation*}

\begin{proof}
Follows at once from definition of $\Omega_{k}^{(v)}$.
\end{proof}
\end{lemma}

This defines symbolic inflation or "splitting and expansion" of $\Omega
_{k}^{(v)}$, some words in $\Omega_{k+1}^{(v)}$ being the same as in $%
\Omega_{k}^{(v)}$ while all the others are "split". The inverse operation is
symbolic deflation or "amalgamation and shrinking", described by a function%
\begin{equation*}
\alpha_{s}:\Omega_{k+1}^{(v)}\rightarrow\Omega_{k}^{(v)}\text{, }\alpha
_{s}(\Omega_{k+1}^{(v)})=\Omega_{k}^{(v)}
\end{equation*}
The operation $\alpha_{s}^{-1}$, whereby $\Omega_{k+1}^{(v)}$ is obtained
from $\Omega_{k}^{(v)}$ by adding symbols to the right-hand end of some
words in $\Omega_{k}^{(v)}$ and leaving other words unaltered, is symbolic
splitting and expansion. In particular, we can define a map $%
\alpha_{s}:\Omega _{k+1}^{(v)}\rightarrow\Omega_{k}^{(v)}$ for all $k\mathbb{%
\in N}_{0}$ according to $\alpha_{s}(\theta)$ is the unique $%
\omega\in\Omega_{k}^{(v)}$ such that $\theta=\omega\beta$ for some $%
\beta\in\Sigma_{\ast}$. Note that $\beta$ may be the empty string. That is,
symbolic amalgamation and shrinking $\alpha_{s}$ is well-defined on $%
\Sigma_{\ast}$.

This tells us that we can use $\Omega_{k}^{(v)}$ to define a partition of $%
\Omega_{m}^{(v)}$ for $m\geq k$. The partition of $\Omega_{k+j}^{(v)}$ is $%
\Omega_{k+j}^{(v)}/\sim$ where $x\sim y$ if $\alpha_{s}^{j}(x)=%
\alpha_{s}^{j}(y)$. To say this another way:

\begin{corollary}
\label{cor:partition}(\textbf{Symbolic Partitions)} For all $m\geq k\geq0,$
the set $\Omega_{k}^{(v)}$ defines a partition $P_{m,k}^{(v)}$ of $\Omega
_{m}^{(v)}$ according to $p\in P_{m,k}^{(v)}$ if and only if there is $%
\omega\in\Sigma_{\ast}$ such that 
\begin{equation*}
p=\{\omega\beta\in\Omega_{m}^{(v)}:\beta\in\Omega_{k}^{(v)}\}.
\end{equation*}
\end{corollary}

\begin{proof}
This follows from Lemma \ref{lemma:split}: for any $\theta\in%
\Omega_{m}^{(v)} $ there is a unique $\omega\in\Omega_{k}^{(v)}$ such that $%
\theta=\omega\beta $ for some $\beta\in\Sigma_{\ast}$. Each word in $%
\Omega_{m}^{(v)}$ is associated with a unique word in $\Omega_{k}^{(v)}$.
Each word in $\Omega _{k}^{(v)}$ is associated with a set of words in $%
\Omega_{m}^{(v)}$.
\end{proof}

According to Lemma \ref{lemma:split}, $\Omega_{k+1}^{(v)}$ may be calculated
by tacking words (some of which may be empty) onto the right-hand end of the
words in $\Omega_{k}^{(v)}$. Now we reverse the description, expressing $%
\Omega_{k}^{(v)}$ as a union of predecessors ($\Omega_{j}^{(v)}$s with $j<k$%
) of $\Omega_{k}^{(v)}$ with words tacked onto their left-hand ends. The
following structural result will reappear (**make explicit) in what follows.

\begin{corollary}
\label{lem:struct}(\textbf{Symbolic Predecessors)} For all $k\geq a_{\max}+l$%
, for all $v\in\mathcal{V}$, for all $l\in\mathbb{N}_{0},$%
\begin{equation*}
\Omega_{k}^{\left( v\right) }=\bigsqcup\limits_{\omega\in\Omega_{l}^{\left(
v\right) }}\omega\Omega_{k-\xi\left( \omega\right) }^{\left( \overrightarrow{%
v}\left( \omega\right) \right) }
\end{equation*}
\end{corollary}

\begin{proof}
It is easy to check that the r.h.s. is contained in the l.h.s.

Conversely, if $\theta\in\Omega_{k}^{\left( v\right) }$ then there is unique 
$\omega\in\Omega_{l}^{\left( v\right) }$ such that $\theta=\omega\beta$ for
some $\beta\in\Sigma_{\ast}$ by Corollary \ref{cor:partition}. Because $%
\omega\beta\in\Sigma_{\ast}$ it follows that $\beta_{1}$ is an edge that
that starts where the last edge in $\omega$ is directed to, namely the
vertex $\overrightarrow{v}\left( \omega\right) $. Finally, since $\xi\left(
\omega\beta\right) =\xi\left( \omega)+\xi(\beta\right) $ it follows that $%
\beta\in\Omega_{k-\xi\left( \omega\right) }^{\overrightarrow{v}\left(
\omega\right) }$.
\end{proof}

\section{\label{canonical}Canonical tilings and their relationship to $\Pi(%
\protect\theta)$}

\begin{definition}
We define sequences of tilings by 
\begin{equation*}
T_{k}=s^{-k}\pi(\Omega_{k}),\text{ }T_{k}^{v}:=s^{-k}\pi(\Omega_{k}^{v})
\end{equation*}
$k\in\mathbb{N},v\in\mathcal{V}$, to be called the \textbf{canonical tilings 
}of the TIFS $(\mathcal{F},\mathcal{G)}.$
\end{definition}

A canonical tiling may be written as a disjoint union of copies of other
canonical tilings. By a copy of a tiling $T$ we mean $ET$ for some $E\in 
\mathcal{T}$, where $\mathcal{T}$ is the set of all isometries of $\mathbb{R}%
^{M}$ generated by the functions of $\mathcal{F}$ together with
multiplication by $s.$

\begin{lemma}
\label{lem:canonical}For all $k\geq a_{\max}+l$, for all $l\in\mathbb{N}%
_{0}, $ for all $v\in\lbrack N]$ 
\begin{equation*}
T_{k}^{v}=\bigsqcup\limits_{\omega\in\Omega_{l}^{v}}E_{k,\omega}T_{k-\xi%
\left( \omega\right) }^{\overrightarrow{v}\left( \omega\right) }\text{ and }%
T_{k}=\bigsqcup\limits_{\omega\in\Omega_{l}}E_{k,\omega}T_{k-\xi\left(
\omega\right) }^{\overrightarrow{v}\left( \omega\right) }
\end{equation*}
where $E_{k,\omega}=s^{-k}f_{\omega}s^{k-e(\omega)}\in\mathcal{T}$ is an
isometry.
\end{lemma}

\begin{proof}
Direct calculation.
\end{proof}

\begin{theorem}
\label{piusingts}For all $\theta\in\Sigma_{\ast}^{\dag},$%
\begin{equation*}
\Pi(\theta)=E_{\theta}T_{\xi(\theta)}^{\overrightarrow{v}(\theta_{\left\vert
\theta\right\vert })},
\end{equation*}
where $E_{\theta}=f_{-\theta}s^{\xi(\theta)}$. Also if $l\in\mathbb{N}_{0},$
and $\xi(\theta)\geq a_{\max}+l$, then%
\begin{equation*}
\Pi(\theta)=\bigsqcup\limits_{\omega\in\Omega_{l}^{\overrightarrow{v}%
(\theta_{\left\vert \theta\right\vert })}}E_{\theta,\omega}T_{\xi(\theta
)-\xi\left( \omega\right) }^{\overleftarrow{v}\left( \omega\right) }
\end{equation*}
where $E_{\theta,\omega}=f_{_{-\theta}}f_{\omega}s^{\xi(\theta)-e(\omega)}$
is an isometry.
\end{theorem}

\begin{proof}
Writing $\theta=\theta_{1}\theta_{2}...\theta_{k}$ so that $\left\vert
\theta\right\vert =k,$ we have from the definitions 
\begin{align*}
\Pi(\theta_{1}\theta_{2}...\theta_{k}) & =f_{_{-\theta_{1}\theta
_{2}...\theta_{k}}}\{\pi\left( \sigma\right) :\sigma\in\Omega_{\xi
(\theta_{1}\theta_{2}...\theta_{k})}^{\overrightarrow{\upsilon}%
(\theta_{k})}\} \\
& =f_{_{-\theta_{1}\theta_{2}...\theta_{k}}}s^{\xi(\theta_{1}\theta
_{2}...\theta_{k})}s^{-\xi(\theta_{1}\theta_{2}...\theta_{k})}\{\pi
(\sigma):\sigma\in\Omega_{\xi(\theta_{1}\theta_{2}...\theta_{k})}^{%
\overrightarrow{\upsilon}(\theta_{k})}\} \\
&
=E_{\theta_{1}\theta_{2}...\theta_{k}}T_{\xi(\theta_{1}\theta_{2}...%
\theta_{k})}^{\overrightarrow{\upsilon}(\theta_{k})}
\end{align*}
which demonstrates that $\Pi(\theta)=E_{\theta}T_{\xi(\theta)}^{%
\overrightarrow{\upsilon}(\theta_{\left\vert \theta\right\vert })}$where $%
E_{\theta}=f_{-\theta}s^{\xi(\theta)}$.

The last statement of the theorem follows similarly from Lemma \ref%
{lem:canonical}.
\end{proof}

\section{All tilings in $\mathbb{T}^{\infty}$ are quasiperiodic}

We recall from \cite{barnsleyvince} the following definitions. A subset $P$
of a tiling $T$ is called a \textit{patch} of $T$ if it is contained in a
ball of finite radius. A tiling $T$ is \textit{quasiperiodic} (also called
repetitive) if, for any patch $P$, there is a number $R>0$ such that any
disk centered at a point in the support of $T$ and is of radius $R$ contains
an isometric copy of $P$. Two tilings are \textit{locally isomorphic} if any
patch in either tiling also appears in the other tiling. A tiling $T$ is%
\textit{\ self-similar} if there is a similitude $\psi$ such that $\psi(t)$
is a union of tiles in $T$ for all $t\in T$. Such a map $\psi$ is called a 
\textit{self-similarity}.

\begin{theorem}
\label{theoremTHREE} Let $(\mathcal{F},\mathcal{G)}$ be a tiling IFS.

\begin{enumerate}
\item Each tiling in $\mathbb{T}_{\infty}$ is quasiperiodic and each pair of
tilings in $\mathbb{T}_{\infty}$ are locally isomorphic.

\item If $\theta\in{\Sigma}_{\infty}^{\dag}$ is eventually periodic, then $%
\Pi(\theta)$ is self-similar. In fact, if $\theta=\alpha\overline{\beta}$
for some $\alpha,\beta\in{\Sigma}_{\infty}^{\dag}$ then $f_{-\alpha}f_{-%
\beta }\left( f_{-\alpha}\right) ^{-1}\Pi(\theta)$ is a self-similarity.
\end{enumerate}
\end{theorem}

\begin{proof}
This uses Theorem \ref{piusingts}, and follows similar lines to \cite[proof
of Theorem 2]{barnsleyvince}.
\end{proof}

\section{\label{localsec}Addresses}

Addresses, both relative and absolute, are described in \cite{barnsleyvince}
for the case $\left\vert \mathcal{V}\right\vert =1$. See also \cite{bandt2}.
Here we add information, and generalize. The relationship between these two
types of addresses is subtle and central to our proof of Theorem \ref{main}.

\subsection{Relative addresses}

\begin{definition}
\label{relativedef}The \textbf{relative address} of $t\in T_{k}^{(v)}$ is
defined to be $\varnothing.\pi^{-1}s^{k}(t)\in\varnothing.\Omega_{k}^{(v)}$.
The relative address\textit{\ }of a tile $t\in T_{k}$ depends on its
context, its location relative to $T_{k}$, and depends in particular on $%
k\in \mathbb{N}_{0}$. Relative addresses also apply to the tiles of $%
\Pi(\theta)$ for each $\theta\in\Sigma_{\ast}^{\dag}$ because $%
\Pi(\theta)=E_{\theta}T_{\xi(\theta)}^{\overleftarrow{\upsilon}%
(\theta_{\left\vert \theta\right\vert })}$ where $E_{\theta}=f_{-\theta}s^{%
\xi(\theta)}$ (by Theorem \ref{piusingts}) is a known isometry applied to $%
T_{\xi(\theta)}$. Thus, the relative address of $t\in\Pi(\theta)$ (relative
to $\Pi(\theta)$) is $\varnothing.\pi ^{-1}f_{-\theta}^{-1}(t)$, for $%
\theta\in\Sigma_{\ast}^{\dag}$.
\end{definition}

\begin{lemma}
The tiles of $T_{k}$ are in bijective correspondence with the set of
relative addresses $\varnothing.\Omega_{k}$. Also the tiles of $T_{k}^{v}$
are in bijective correspondence with the set of relative addresses $%
\varnothing .\Omega_{k}^{v}$.
\end{lemma}

\begin{proof}
We have $T_{k}=s^{-k}\pi(\Omega_{k})$ so $s^{-k}\pi$ maps $\Omega_{k}$ onto $%
T_{k}$. Also the map $s^{-k}\pi:\Omega_{k}\rightarrow T_{k}$ is one-to-one:
if $\beta\neq\gamma,$ for $\beta,\gamma\in\Sigma_{\ast}$ then $f_{\beta
}(A)\neq f_{\gamma}(A)$ because $t=$ $s^{-k}\pi(\beta)=s^{-k}\pi(\gamma)$%
with $\beta,\gamma\in T_{k}$ implies $\beta=\gamma$.
\end{proof}

For precision we should write "the relative address of $t$ relative to $%
T_{k} $" or equivalent: however, when the context $t\in T_{k}$ is clear, we
may simply refer to "the relative address of $t$".

\begin{example}
\label{ex01}(Standard 1D binary tiling) For the IFS $\mathcal{F}_{0}=\{%
\mathbb{R};f_{1},f_{2}\}$ with $f_{1}(x)=0.5x,f_{2}(x)=$ $0.5x+0.5$ we have $%
\Pi(\theta)$ for $\theta\in\Sigma_{\ast}^{\dag}$ is a tiling by copies of
the tile $t=[0,0.5]$ whose union is an interval of length $2^{\left\vert
\theta\right\vert }$ and is isometric to $T_{\left\vert \theta\right\vert }$
and represented by $tttt....t$ with relative addresses in order from left to
right 
\begin{equation*}
\varnothing.111...11,\varnothing.111...12,\varnothing
.111...21,....,\varnothing.222...22,
\end{equation*}
the length of each string (address) being $\left\vert \theta\right\vert +1.$
Notice that here $T_{k}$ contains $2^{\left\vert \theta\right\vert }-1$
copies of $T_{0}$ (namely tt) where a copy is $ET_{0}$ where $E\in\mathcal{T}%
_{\mathcal{F}_{0}}$, the group of isometries generated by the functions of $%
\mathcal{F}_{0}$.
\end{example}

\begin{example}
\label{ex02}(Fibonacci 1D tilings) $\mathcal{F}_{1}\mathcal{=\{}%
ax,a^{2}x+1-a^{2},a+a^{2}=1,$ $a>0\},$ $\mathcal{T=T}_{\mathcal{F}_{1}}$ is
the largest group of isometries generated by $\mathcal{F}_{1}$. The tiles of 
$\Pi(\theta)$ for $\theta\in\Sigma_{\ast}^{\dag}$ are isometries that belong
to $\mathcal{T}_{\mathcal{F}_{1}}$ applied to the tiling of [0,1] provided
by the IFS, writing the tiling $\Pi(\varnothing)=T_{0}$ as $ls$ where $l$ is
a copy of $[0,a]$ and (here) $s$ is a copy of $[0,a^{2}]$ we have:

$T_{0}=ls$ has relative addresses $\varnothing.1,\varnothing.2$ (i.e. the
address of $l$ is $1$ and of $s$ is $2$)

$T_{1}=lsl$ has relative addresses $\varnothing.11,\varnothing.12,%
\varnothing .2$

$T_{2}=lslls$ has relative addresses $\varnothing.111,\varnothing
.112,\varnothing.12,\varnothing.21,\varnothing.22$

$T_{3}=lsllslsl$ has relative addresses $\varnothing.1111,\varnothing
.1112,\varnothing.112,\varnothing.121,\varnothing.122,\varnothing
.211,\varnothing.212,22$

We remark that $T_{k}$ comprises $F_{k+1}$ distinct tiles and contains
exactly $F_{k}$ copies (under maps of $\mathcal{T}_{\mathcal{F}_{1}}$) of $%
T_{0}$, where $\{F_{k}:k\in\mathbb{N}_{0}\}$ is a sequence of Fibonacci
numbers $\{1,2,3,5,8,13,21,...\}$.

Note that $T_{4}=lsllslsllslls$ contains two "overlapping" copies of $lslls$.
\end{example}

\subsection{Absolute addresses}

The set of \textit{absolute addresses} associated with $(\mathcal{F},%
\mathcal{G)}$ is 
\begin{equation*}
\mathbb{A}:=\{\theta.\omega:\theta\in\Sigma_{\ast}^{\dag},\,\omega\in
\Omega_{\xi(\theta)}^{\overleftarrow{\upsilon}(\theta_{\left\vert
\theta\right\vert })},\,\theta_{\left\vert \theta\right\vert }\neq\omega
_{1}\}.
\end{equation*}
Define $\widehat{\pi}:\mathbb{A\rightarrow\{}t\in T:T\in\mathbb{T\}}$ by 
\begin{equation*}
\widehat{\pi}(\theta.\omega)=f_{-\theta}.f_{\omega}(A).
\end{equation*}
The condition $\theta_{\left\vert \theta\right\vert }\neq\omega_{1}$ is
imposed. We say that $\theta.\omega$ is an \textit{absolute address} of the
tile $f_{-\theta}.f_{\omega}(A)$. It follows from Definition \ref{defONE}
that the map $\widehat{\pi}$ is surjective: every tile of $\mathbb{\{}t\in
T:T\in\mathbb{T\}}$ possesses at least one absolute address.

In general a tile may have many different absolute addresses. The tile $%
[1,1.5]$ of Example \ref{ex01} has the two absolute addresses $1.21$ and $%
21.211$.

\subsection{Relationship between relative and absolute addresses}

\begin{theorem}
If $t\in\Pi(\theta)\backslash\Pi(\varnothing)$ with $\theta\in\Sigma_{\ast
}^{\dag}$ has relative address $\omega$ relative to $\Pi(\theta),$ then an
absolute address of $t$ is $\theta_{1}\theta_{2}...\theta_{l}.S^{\left\vert
\theta\right\vert -l}\omega$ where $l\in\mathbb{N}$ is the unique index such
that 
\begin{equation}
t\in\Pi(\theta_{1}\theta_{2}...\theta_{l})\text{ and }t\notin\Pi(\theta
_{1}\theta_{2}...\theta_{l-1})  \label{*eqn}
\end{equation}

(If $t\in\Pi(\varnothing)$, then the unique absolute address of $t$ is $%
\emptyset.n$ for some $n\in N$.)
\end{theorem}

\begin{proof}
Recalling that%
\begin{equation*}
\Pi(\varnothing)\subset\Pi(\theta_{1})\subset\Pi(\theta_{1}\theta_{2})%
\subset...\subset\Pi(\theta_{1}\theta_{2}...\theta_{\left\vert \theta
\right\vert -1})\subset\Pi(\theta),
\end{equation*}
we have disjoint union 
\begin{equation*}
\Pi(\theta)=\Pi(\varnothing)\cup\left( \Pi(\theta_{1})\backslash
\Pi(\varnothing)\right) \cup\left( \Pi(\theta_{1}\theta_{2})\backslash
\Pi(\theta_{1})\right) \cup...\cup\left( \Pi(\theta)\backslash\Pi(\theta
_{1}\theta_{2}...\theta_{\left\vert \theta\right\vert -1})\right) .
\end{equation*}
So there is a unique $l$ such that Equation (\ref{*eqn}) is true. Since $%
t\in\Pi(\theta)$ has relative address $\omega$ relative to $\Pi(\theta)$ we
have 
\begin{equation*}
\omega=\varnothing.\pi^{-1}f_{-\theta}^{-1}(t)
\end{equation*}
and so an absolute adddress of $t$ is 
\begin{equation*}
\theta.\omega|_{cancel}=\theta.\pi^{-1}f_{-\theta}^{-1}(t)|_{cancel}
\end{equation*}
where $|_{cancel}$ means equal symbols on either side of $"."$ are removed
until there is a different symbol on either side. Since $t\in\Pi(\theta
_{1}\theta_{2}...\theta_{l})$ the terms $\theta_{l+1}\theta_{l+2}...\theta_{%
\left\vert \theta\right\vert }$ must cancel yielding the absolute address%
\begin{equation*}
\theta.\omega|_{cancel}=\theta_{1}\theta_{2}...\theta_{l}.\omega
_{|\theta|-l+1}...\omega_{|\omega|}
\end{equation*}
\end{proof}

\section{\label{rigidsec1}Local rigidity and its consequences}

\subsection{Definition of locally rigid}

\begin{figure}[tbp]
\centering%
\includegraphics[
height=2.0193in,
width=2.0193in.
]{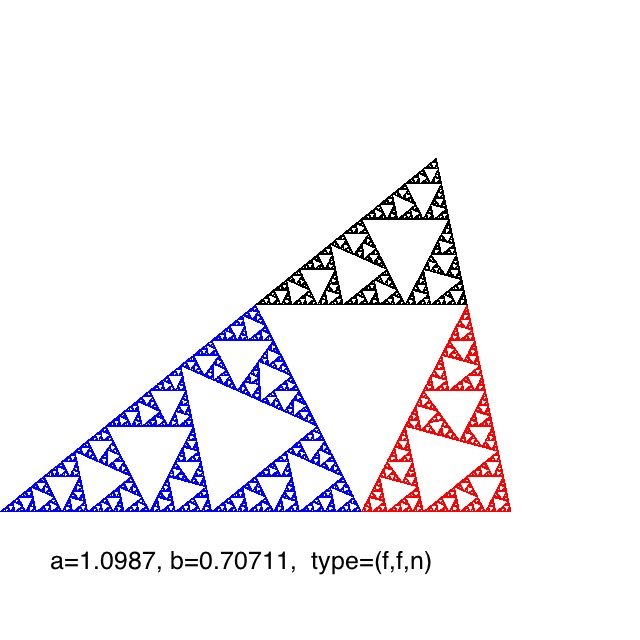}
\caption{Attractor of a rigid TIFS.}
\label{rigid}
\end{figure}

Let $\mathcal{T}$ be the group of isometries generated by the set of maps of 
$\mathcal{F}$, and let $\mathcal{U}$ be the group of all isometries on $%
\mathbb{R}^{M}$. Let $\mathcal{T}^{\prime }\subset \mathcal{T}$ be the
groupoid of isometries of the form $f_{-\theta }f_{\sigma }$ where $\sigma
\in \Sigma _{\ast },\theta \in \Sigma _{\ast }^{\dag }$ and $\overleftarrow{v%
}(\sigma _{1})\in \overleftarrow{v}(\theta _{\left\vert \theta \right\vert
}).$

\begin{definition}
\label{localdef} The family of tilings $\mathbb{T}:=\left\{ \Pi (\theta
):\theta \in \Sigma ^{\dag }\right\} ,$ and the TIFS $(\mathcal{F},\mathcal{%
G)}$, are said to be\textbf{\ locally rigid} when the following two
statements are true: (i) if $E\in \mathcal{T}$ is such that $T_{0}^{v}\cap
E_{2}T_{0}^{w}$ tiles $A^{v}\cap EA^{w}$ then $E=id$ and $v=w$; (ii) there
is only one symmetry of each $A^{v}$ contained in $\mathcal{T}^{\prime }$.

The TIFS $(\mathcal{F},\mathcal{G)}$ is said to be \textbf{rigid} if
statements (i) and (ii) are true when $\mathcal{T}^{\prime }$ is replaced by 
$\mathcal{U}$.
\end{definition}

Figure \ref{rigid} illustrates the attractor of a locally rigid system. The
TIFS in Example \ref{ex02} is not rigid but it is locally rigid. The notion
of a rigid tiling was introduced for the case $\left\vert \mathcal{V}%
\right\vert =1$ in \cite{barnsleyvince}.

\begin{newrigid}

\subsection{Tighter definition of locally rigid}

The following definition is less restrictive than Definition \ref{localdef}.
It is related to the concept of "neighbor maps" \cite{bandtneighbor1,
bandtneighbor2, bandtneighbor3}. Neighbor maps are important because of
their connection to the OSC. **It should probably replace Definition \ref%
{localdef}, which will mean redefining $\mathcal{T}$ as a groupoid in the
set of isometries generated by $\mathcal{F}$.

\begin{definition}
\label{lrigid2}The TIFS $(\mathcal{F},\mathcal{G)}$ is defined to be locally
rigid if two conditions hold:

(i) For all $i\neq j$ there are $k\neq i$ and $l\neq j$ so that (a) 
\begin{equation*}
f_{j}^{-1}f_{k}(A^{\overleftarrow{v}(k)})\cap f_{i}^{-1}f_{l}(A^{%
\overleftarrow{v}(l)})\neq f_{i}^{-1}f_{l}(A^{\overleftarrow{v}(l)})
\end{equation*}%
There are obvious restrictions on $i$ and $j$ implied by $\mathcal{G}$. In
the language of Bandt, this says that $f_{j}^{-1}f_{k}$ and $f_{i}^{-1}f_{l}$
are proper neighbor maps whose images do not agree.

(ii) There is only one symmetry of each $A^{v}$ contained in $\mathcal{T}$.
(This is to ensure that how any copy of $A^{v}$ occuring in any scaled
tiling $f_{i}^{-1}\Pi(\theta)$ ($i$ as allowed by $\mathcal{G}$) is tiled is
unequivocal.) Clearly this condition is met if each $A^{v}$ has only one
symmetry.
\end{definition}

**Examples and discussion could go here.

\begin{theorem}
Let $(\mathcal{F},\mathcal{G)}$ be a tiling iterated function system. Then
the following statements are equivalent.

(I) $(\mathcal{F},\mathcal{G)}$ is locally rigid (Definition \ref{lrigid2})

(II) The inflation function, mapping $\Pi(\theta)$ to $\Pi(S\theta),$ and
the associated inverse maps, mapping $\Pi(\theta)$ to $\Pi(n\theta)$ (as
allowed) are well defined.

(III) $\Pi:\Sigma^{\dag}\rightarrow\mathbb{T}$ is injective.

(IV) $\Pi(\theta)=E\Pi(\psi)$ if and only if $E$ is a neighbor map etc...
\end{theorem}

\begin{proof}
The proof that if a system is locally rigid as here defined, then it is
locally rigid as defined earlier is straightfoward and interesting. So (I)\
implies (II), (III) and (IV) which also imply one-another in various ways
already established. What is interesting is that (III) implies (I).
\end{proof}

\subsection{Flexible}

Restrict attention to the case of a TIFS in case $\left\vert \mathcal{V}%
\right\vert =1$.

\begin{definition}
The TIFS $\mathcal{F}$ is \textbf{inflexible }if for all $i$ and $j$ with $%
i\neq j$ with $a_{i}=a_{j}$ such that there is some $k$ and $l$ with $k\neq
l $ and $a_{k}=a_{l}$ such that 
\begin{equation*}
f_{j}^{-1}f_{k}(A)\cap f_{i}^{-1}f_{l}(A)
\end{equation*}%
is a \textbf{bad} intersection. (If all such intersections are good, i.e.
integral tiles or sets of lower dimension than that of the attractor, then
the TIFS\ is \textbf{flexible}.
\end{definition}

\end{newrigid}

\subsection{Main Theorem}

\begin{newrigid}

** This is continuation from before the previous section.

\end{newrigid}

We define $\mathbb{Q}:=\{ET:E\in\mathcal{T}$, $T\in\mathbb{T\}}$ and $%
\mathbb{Q}^{\prime}:=\{ET:E\in\mathcal{T}$, $T\in\mathbb{T},\mathbb{\ }T\neq
T_{0}^{v},v\in\mathcal{V}\mathbb{\}}$.

\begin{definition}
Let $\mathcal{F}$ be a locally rigid IFS. Any tile in $\mathbb{Q}$ that is
isometric to $s^{a_{\max}}A^{v}$ is called a \textbf{small tile}, and any
tile that is isometric to $sA^{v}$ is called a \textbf{large tile}. We say
that a tiling $P\in\mathbb{Q}$ comprises a set of \textbf{partners }if $%
P=ET_{0}^{v}$ for some $E\in\mathcal{T}$, $v\in\mathcal{V}$. Given $Q\in%
\mathbb{Q}$ we define $partners(Q)$ to be the set of all sets of partners in 
$Q$.
\end{definition}

The terminology of large and small tiles is useful in discussing some
examples. If a tiling $T\in\mathbb{Q}$ is locally rigid, then each set of
partners in $T$ has no partners in common with any other set of partners in $%
T$.

Define for convenience: 
\begin{align*}
\Lambda_{k}^{v} & =\{\sigma\in\Sigma_{\ast}:\xi(\sigma)=k,\overleftarrow{v}%
(\sigma_{1})=v\}\subset\Omega_{k-1}^{v} \\
\Lambda_{k} & =\cup_{v}\Lambda_{k}^{v}\subset\Omega_{k-1}
\end{align*}

\begin{theorem}
\label{keythm}Let $\mathcal{F}$ be locally rigid and let $T_{k}$ be given.

(i) There is a bijective correspondence between $\Lambda_{k}^{v}$ and the
set of copies $ET_{0}^{v}\subset T_{k}$ with $E\in\mathcal{T}$.

(ii) If $ET_{0}^{v}\subset T_{k}$ for some $E\in\mathcal{T}$, then there is
unique $\sigma\in\Lambda_{k}^{v}$ such that 
\begin{equation*}
E=E_{\sigma_{\left\vert \sigma\right\vert }\sigma_{\left\vert \sigma
\right\vert -1}...\sigma_{1}}^{-1}=(f_{-\sigma_{\left\vert \sigma\right\vert
}\sigma_{\left\vert \sigma\right\vert
-1}...\sigma_{1}}s^{k})^{-1}=s^{-k}f_{\sigma}
\end{equation*}
\end{theorem}

\begin{proof}
(i) If $(\mathcal{F},\mathcal{G})$ is locally rigid, then given the tiling $%
ET_{k}$ with $E\in\mathcal{T}$ we can identify $E$ uniquely. The relative
addresses of tiles in $ET_{k}$ may then be calculated in tandem by repeated
application of $\alpha^{-1}$. Each tile in $ET_{0}$ is associated with a
unique relative addresses in $\Omega_{0}=[N]$. Now assume that, for all $%
l=0,1,...,k,$ we have identified the tiles of $ET_{l}$ with their relative
addresses (relative to $T_{l}$). These lie in $\Omega_{l}$. Then the
relative addresses of the tiles of $ET_{k+1}$ (relative to $T_{k+1})$ may be
calculated from those of $ET_{k}$ by constructing the set of sets $%
s^{-1}ET_{k}$, and then splitting the images of large tiles$,$ namely those
that are of the form $s^{-1}FA^{v}$ for some $v\in\mathcal{V}$ and $F\in%
\mathcal{T}$ , to form nonintersecting sets of partners of the form $%
\{Ff_{i}(A^{\overleftarrow{v}(i)}):i\in\mathcal{E}_{v,\ast}\}$, assigning to
these "children of the split" the relative addresses of their parents
(relative to $T_{k}$) together with an additional symbol $i\in\left[ N\right]
$ added on the right-hand end according to its relative address relative to
the copy of $T_{0}$ to which it belongs. By local rigidity, this can be done
uniquely. The relative addresses (relative to $T_{k+1}$) of the tiles in $%
s^{-1}ET_{k}$ that are not split and so are simply $s^{-1}$ times as large
as their predecessors, are the same as the relative addresses of their
predecessors relative to $T_{k}.$

(ii) It follows in particular that if $\mathcal{F}$ is locally rigid and $%
E^{\prime}T_{0}\subset T_{k}$, then the relative addresses of the tiles of $%
E^{\prime}T_{0}$ must be $\{\varnothing.\sigma_{1}...\sigma_{\left\vert
\sigma\right\vert }i:i\in\lbrack N]\}$ for some $\sigma_{1}...\sigma
_{\left\vert \sigma\right\vert }\in\Sigma_{\ast}$ with $\xi(\sigma
_{1}...\sigma_{\left\vert \sigma\right\vert })=k$. In this case we say that
the relative address of $E^{\prime}T_{0}$ (relative to $T_{k}$) is $%
\varnothing.\sigma_{1}...\sigma_{\left\vert \sigma\right\vert }$.
\end{proof}

\begin{theorem}
\label{main}Let $(\mathcal{F},\mathcal{G)}$ be locally rigid. Then $\Pi
(\theta)=E\Pi(\psi)$ for some $E\in\mathcal{T}$, $\theta,\psi\in\Sigma^{%
\dag} $if and only if there are $p,q\in\mathbb{N}_{0}$ such that $\xi(\theta
|p)=\xi(\psi|q),$ $E=E_{\theta|p}E_{\psi|q}^{-1}$ and $S^{p}\theta=S^{q}\psi$%
.
\end{theorem}

\begin{proof}
If there are $p,q\in\mathbb{N}_{0}$ such that $\xi(\theta|p)=\xi(\psi|q),$ $%
E=E_{\theta|p}E_{\psi|q}^{-1}$ and $S^{p}\theta=S^{q}\psi\theta\psi$%
\begin{align*}
\Pi(\theta) & =\bigsqcup\limits_{m\in\mathbb{N}_{0}}f_{-\theta|\left(
p+m\right) }s^{\xi\left( \theta|\left( p+m\right) \right) }T_{\xi\left(
\theta|\left( p+m\right) \right) } \\
& =f_{-\theta|p}\bigsqcup\limits_{m\in\mathbb{N}_{0}}f_{-\psi_{q+1}\psi
_{q+2}...\psi_{q+m}}s^{\xi\left( \psi_{q+1}\psi_{q+2}...\psi_{q+m}\right)
}s^{\xi(\psi|q)}T_{\xi(\psi|q+m)} \\
& =f_{-\theta|p}f_{-\psi|q}^{-1}\bigsqcup\limits_{m\in\mathbb{N}%
_{0}}f_{-\psi|\left( q+m\right) }s^{\xi\left( \psi|\left( q+m\right) \right)
}T_{\xi\left( \psi|\left( q+m\right) \right) } \\
& =E_{\theta|p}E_{\psi|q}^{-1}\Pi(\psi)
\end{align*}
This completes the proof in one direction.

To prove the converse we suppose that $\mathcal{F}$ is locally rigid and
that $\Pi(\theta)=E\Pi(\psi)$ for some $E\in\mathcal{T}$, where $\theta,\psi
\in\Sigma^{\dag}$. Let $m$ be any integer such that $E\Pi(\varnothing
)\subset\Pi(\theta|m).$ It follows that 
\begin{equation*}
E_{\theta|m}^{-1}E\Pi(\varnothing)\subset T_{\xi(\theta|m)}
\end{equation*}
Then by Theorem \ref{keythm} (ii) the set of relative addresses (relative to 
$T_{\xi(\theta|m)}$) of copies of $T_{0}$ in $\Pi(\theta|m)$ is 
\begin{equation*}
\left\{ \sigma i:i\in\lbrack N],\sigma\in\Sigma_{\ast},\xi(\sigma)=\xi
(\theta|m)\right\} .
\end{equation*}
It follows that $E_{\theta|m}^{-1}E\Pi(\varnothing)=E_{\sigma_{|\sigma|}%
\sigma_{|\sigma|-1}...\sigma_{1}}^{-1}\Pi(\varnothing)$ for some unique $%
\sigma=\sigma_{1}\sigma_{2}...\sigma_{|\sigma|-1}\sigma_{|\sigma|}$ such
that $\xi(\sigma)=\xi(\theta|m).$ It follows that 
\begin{equation*}
E=E_{\theta|m}E_{\sigma_{|\sigma|}\sigma_{|\sigma|-1}...\sigma_{1}}^{-1},
\end{equation*}
where we have used local rigidity. We know the absolute addresses of the
tiles of $\Pi(\varnothing)\subset\Pi(\theta|m)$ are 
\begin{equation*}
\left\{ \theta_{1}\theta_{2}...\theta_{m}.\theta_{m}...\theta_{2}\theta
_{1}i|_{cancel}\text{ for }i\in\lbrack N]\right\} .
\end{equation*}
Given $E=E_{\theta|m}E_{\sigma_{|\sigma|}\sigma_{|\sigma|-1}...%
\sigma_{1}}^{-1},$ the absolute addresses of $E\Pi(\varnothing)\subset\Pi(%
\theta|m)$ are then $\left\{ \theta_{1}\theta_{2}...\theta_{m}.\sigma
i|_{cancel}\text{ for }i\in\lbrack N]\right\} .$ Since $\Pi(\theta)=E\Pi(%
\psi),$ 
\begin{equation*}
\psi_{1}\psi_{2}...\psi_{\left\vert \sigma\right\vert }=\sigma_{\left\vert
\sigma\right\vert }\sigma_{\left\vert \sigma\right\vert -1}...\sigma_{1}
\end{equation*}
and thus 
\begin{equation*}
E=E_{\theta|m}E_{\psi_{1}\psi_{2}...\psi_{\left\vert \sigma\right\vert
}}^{-1}
\end{equation*}
where $\xi\left( \psi_{1}\psi_{2}...\psi_{\left\vert \sigma\right\vert
}\right) =\xi\left( \sigma\right) =\xi(\theta|m)$.

Now let $k\in\mathbb{N}$ and consider the two sets $\Pi(\varnothing)$ and $%
E\Pi(\varnothing)$ both of which belong to $\Pi(\theta|m)=E_{\theta|m}T_{%
\xi(\theta|m)}$ which in turn is contained in $\Pi(\theta|m+k)=E_{\theta
|\left( m+k\right) }T_{\xi(\theta|\left( m+k\right) )}$. We are going to 
\textit{calculate the relative addresses of both} $\Pi(\varnothing)$ \textit{%
and} $E\Pi(\varnothing)$ \textit{relative to} $T_{\xi(\theta|\left(
m+k\right) )}$ \textit{in terms of their relative addresses relative to} $%
T_{\xi(\theta|m)}$. Using Definition \ref{relativedef} we find: the relative
address of $t\in\Pi(\theta|m)\subset\Pi(\theta|m+k)$ relative to $%
T_{\xi(\theta|m)}$ is $\omega=\pi^{-1}(s^{\xi(\theta|m)}E_{\theta|m}^{-1}t)$
and relative to $T_{\xi(\theta|\left( m+k\right) )}$ it is $\widetilde{\omega%
}=\pi^{-1}(s^{\xi(\theta|\left( m+k\right) )}E_{\theta|\left( m+k\right)
}^{-1}t)$. It follows that $\widetilde{\omega }=\theta_{m+k}%
\theta_{m+k-1}...\theta_{m+1}\omega.$ Hence the relative addresses of $%
\Pi(\varnothing)$ and $E\Pi(\varnothing)$ relative to $T_{\xi(\theta|\left(
m+k\right) )}$ are $\varnothing.\theta_{m+k}\theta_{m+k-1}...\theta_{m+1}%
\theta_{m}...\theta_{1}$ and $\varnothing
.\theta_{m+k}\theta_{m+k-1}...\theta_{m+1}\psi_{\left\vert \sigma\right\vert
}\psi_{\left\vert \sigma\right\vert -1}...\psi_{1}$. It follows that $%
S^{\left\vert \sigma\right\vert }\psi=S^{m}\theta$.
\end{proof}

\begin{corollary}
\label{cor01}If $\mathcal{(F},\mathcal{V)}$ is locally rigid, then $\Pi
(\theta)=E\Pi(\theta)$ if and only if $E=id$.
\end{corollary}

\begin{corollary}
\label{cor02}If $\mathcal{(F},\mathcal{V)}$ is locally rigid, then $\Pi
:\Sigma^{\dag}\rightarrow\mathbb{T}$ is a homeomorphism.
\end{corollary}

\section{Inflation and deflation\label{inflation}}

If $(\mathcal{F},\mathcal{G)}$ is locally rigid, then the operations of
inflation or "expansion and splitting" of tilings in $\mathbb{Q}$, and
deflation or "amalgamation and shrinking" of tilings in $\mathbb{Q}^{\prime}$
are well-defined. We handle these concepts with the operators $\alpha^{-1}$
and its inverse $\alpha,$ respectively, also used in \cite{barnsleyvince}.

\begin{theorem}
Let $\mathcal{F}$ be a locally rigid IFS. The amalgamation and shrinking
(deflation) operation $\alpha:\mathbb{Q}^{\prime}\mathbb{\rightarrow Q}$ is
well-defined by 
\begin{equation*}
\alpha Q^{\prime}=\{st:t\in Q\backslash partners(Q^{\prime})\}\cup
\bigsqcup\{sEA^{v}:E\in\mathcal{T},ET_{0}^{v}\subset
partners(Q^{\prime}),v\in\mathcal{V}\}
\end{equation*}
for all $Q^{\prime}\in\mathbb{Q}^{\prime}$. The expansion and splitting
(inflation) operator $\alpha:\mathbb{Q\rightarrow Q}$ is well-defined by%
\begin{equation*}
\alpha Q=\{s^{-1}t:t\text{ is not congruent to }sA\}\cup\bigsqcup
\{sET_{0}:E\in\mathcal{T},sEA\in Q\}
\end{equation*}
fro all $Q\in\mathbb{Q}$. In particular, $\alpha T_{k}=T_{k-1}$ and $%
\alpha^{-1}T_{k-1}=T_{k}$ for all $k\in\mathbb{N}$.
\end{theorem}

\begin{proof}
We explained essentially this in the proof of Theorem \ref{keythm}(i). See
also \cite[Lemmas 6 and 7]{barnsleyvince} with "rigid" replaced by "locally
rigid".
\end{proof}

\begin{theorem}
\label{keythm2}Let $\mathcal{F}$ be locally rigid and let $T_{k}$ be given.

(i) The following \textbf{hierarchy of} $\sigma\in\Sigma_{\ast}$ obtains: 
\begin{equation}
ET_{0}^{v}\subset F_{1}T_{\xi(S^{|\sigma|-1}\sigma)}\subset
F_{2}T_{\xi(S^{|\sigma|-2}\sigma)}\subset...\subset F_{\left\vert
\sigma\right\vert -1}T_{\xi(S\sigma)}\subset T_{k=\xi(\sigma)}
\label{heirarchy}
\end{equation}
where $F_{j}=s^{-\xi(S^{|\sigma|-j}\sigma)}E_{\sigma_{\left\vert
\sigma\right\vert -j}...\sigma_{1}}^{-1}s^{\xi(S^{|\sigma|-j}\sigma)}$ and $%
E_{\theta}$ is the isometry $f_{-\theta}s^{\xi(\theta|\left\vert
\theta\right\vert )}.$ Application of $\alpha^{\xi(\sigma_{\left\vert
\sigma\right\vert })}$ to the hierarchy of $\sigma_{1}...\sigma_{\left\vert
\sigma\right\vert }$ minus the leftmost inclusion yields the heirarchy of $%
\sigma_{1}...\sigma_{\left\vert \sigma\right\vert -1}.$

(ii) For all $\theta\in\Sigma_{\infty}^{\dag}$, $n\in\left[ N\right] ,$ $k\in%
\mathbb{N}_{0},$%
\begin{equation*}
\alpha^{\xi(\theta|k)}E_{\theta|k}^{-1}\Pi(\theta)=\Pi(S^{k}\theta)\text{
and }\alpha^{-a_{n}}\Pi(\theta)=s^{-a_{n}}f_{n}\Pi(n\theta)
\end{equation*}
where $E_{\theta|k}=f_{-\theta|k}s^{\xi(\theta|k)}$.
\end{theorem}

\begin{proof}
(i) Equation \ref{heirarchy} is the result of applying $E_{\sigma_{\left%
\vert \sigma\right\vert }\sigma_{\left\vert \sigma\right\vert
-1}...\sigma_{1}}^{-1}$ to the chain of inclusions 
\begin{equation*}
T_{0}=\Pi(\varnothing)\subset\Pi(\sigma_{\left\vert \sigma\right\vert
})\subset\Pi(\sigma_{\left\vert \sigma\right\vert }\sigma_{\left\vert
\sigma\right\vert -1})\subset...\subset\Pi(\sigma_{\left\vert \sigma
\right\vert }\sigma_{\left\vert \sigma\right\vert -1}...\sigma_{2})\subset
\Pi(\sigma_{\left\vert \sigma\right\vert }\sigma_{\left\vert
\sigma\right\vert -1}...\sigma_{1})
\end{equation*}
where we recall that $\Pi(\theta)=E_{\theta}T_{\xi(\theta)}^{\overrightarrow{%
\upsilon}(\theta_{\left\vert \theta\right\vert })}$ (Theorem \ref{piusingts}%
) for all $\theta\in\Sigma_{\ast}^{\dag},$ where $E_{\theta
}:=f_{-\theta}s^{\xi(\theta)}$.

(ii) This follows from $\Pi(\theta)=E_{\theta}T_{\xi(\theta)}^{%
\overrightarrow{\upsilon}(\theta_{\left\vert \theta\right\vert })}$ and $%
\alpha T=sTs^{-1}\alpha$ for any $T:\mathbb{R}^{M}\rightarrow\mathbb{R}^{M}$.
\end{proof}

Taking $k=1$ in (ii) we have%
\begin{equation*}
\alpha^{a_{\theta_{1}}}\Pi(\theta)=s^{a_{\theta_{1}}}f_{\theta_{1}}^{-1}%
\Pi(S\theta)\text{ and }\alpha^{-a_{n}}\Pi(\theta)=s^{-a_{n}}f_{n}\Pi
(n\theta)\text{.}
\end{equation*}
Because $\Pi$ is one-to-one when $(\mathcal{F},\mathcal{G)}$ is locally
rigid, this implies:\ \textit{Given the tiling }$\Pi(\theta)$ it is possible
to: (I)\textit{\ Determine }$\theta_{1}$\textit{\ and therefore }$\theta$%
\textit{\ by means of a sequence of geometrical tests and to calculate }$%
\Pi(S^{a_{\theta _{1}}}\theta),$ essentially by applying $\alpha$ the right
number of times and then applying the appropriate isometry; (II)\textit{\
Transform }$\Pi(\theta )$\textit{\ to }$\Pi(n\theta)$\textit{\ for any }$%
n\in\lbrack N]$,\textit{\ }by applying\textit{\ }$\alpha^{-a_{n}}$
(inflation $a_{n}$ times) and then applying the isometry $s^{-a_{n}}f_{n}$.

\section{Dynamics on tiling spaces}

Here we focus on the situation in \cite{barnsleyvince} where $\left\vert 
\mathcal{V}\right\vert =1$. It appears that the ideas go through in the
general case. Recall that 
\begin{equation*}
\mathbb{T}=\{\Pi(\mathbb{\theta)}:\theta\in\Sigma^{\dag}\}\text{ and }%
\mathbb{T}_{\infty}=\{\Pi(\mathbb{\theta)}:\theta\in\Sigma_{\infty}^{\dag}\}
\end{equation*}
We consider the structure and the action of the inflation/deflation
dynamical system on each of the following two spaces. We restrict attention
to $\left( \mathcal{F},\mathcal{G}\right) $ being locally rigid.

(1) The tiling space is%
\begin{equation*}
\widetilde{\mathbb{T}}:=\mathbb{T}_{\infty}/\sim
\end{equation*}
where $\Pi(\theta)\sim\Pi(\psi)$ iff $E_{1}\Pi(\theta)=E_{2}\Pi(\psi)$ for
some $E_{1},$ $E_{2}\in\mathcal{T}$. Here we assume that $\mathcal{T}$ is
the group generated by the set of isometries that map from the prototile set
to the tilings $\mathbb{T}_{\infty}$. $\mathcal{T}$ may be replaced by any
larger group. Each member of $\widetilde{\mathbb{T}}$ has a representative
in $\mathbb{T}_{\infty}$. We denote the equivalence class of $\Pi(\theta)$
by $\left[ \Pi(\theta)\right] $. In the absence of anything cleverer, the
topology of $\widetilde{\mathbb{T}}$ is the discrete topology.

EXAMPLES: (i) (Fibonacci 1D tilings) $\mathcal{F}_{1}\mathcal{=\{}%
ax,a^{2}x+1-a^{2},a+a^{2}=1,$ $a>0\},$ $\mathcal{T}$ is the set of
1D-translations, or a subgroup of this set of translations, such that any
tiling in $\Pi_{_{1}}(\theta)$ is a union of tiles of the form $gt$ for some 
$g\in\mathcal{T}$ and $t\in\mathcal{P}_{1}=\{[0,a],[a,1]\}$.

(ii) $\mathcal{F}_{2}$ is the golden b IFS described elsewhere. It comprises
two maps and two prototiles. In this case $\mathcal{T=T}_{2}$ is any group
of isometries on $\mathbb{R}^{2}$ that contains pair of isometries%
\begin{equation*}
\begin{pmatrix}
0 & -1 \\ 
1 & 0%
\end{pmatrix}%
\begin{pmatrix}
x \\ 
y%
\end{pmatrix}
+%
\begin{pmatrix}
1 \\ 
0%
\end{pmatrix}
,%
\begin{pmatrix}
1 & 0 \\ 
0 & -1%
\end{pmatrix}%
\begin{pmatrix}
x \\ 
y%
\end{pmatrix}
+%
\begin{pmatrix}
0 \\ 
1%
\end{pmatrix}%
\end{equation*}

(iii) $\mathcal{F}_{3}$ is a different golden $b$ IFS comprising I think 13
maps, in anycase more than two maps. Each map is obtained by composing maps
of $\mathcal{F}_{2}$. The prototile set $\mathcal{P}_{3}$ comprises eight
prototiles and $\mathcal{T}=\mathcal{T}_{3}$ is for example the group of
translations on $\mathbb{R}^{2}$. The set of tilings of in this case are
essentially the same as in the case (ii) but the addressing structure is
different.

(2) The tiling space is 
\begin{equation*}
\widehat{\mathbb{T}}=(\mathbb{T}_{\infty}\mathbb{\times}\mathcal{T)}/\sim
\end{equation*}
where $\mathbb{T}_{\infty}\mathbb{\times}\mathcal{T}$ is equipped with the
metric $d_{\mathbb{T}_{\infty}}+d_{\mathcal{T}}$ and $\Pi(\theta)\times
E_{1}\sim\Pi(\psi)\times E_{2}$ iff $E_{1}\Pi(\theta)=E_{2}\Pi(\psi)$ with
the induced metric. This is the tiling space considered, for example by
Anderson and Putnam and many others. It is relevant to spectral analysis of
tilings and, in cases where $A$ is a polygon, to interval exchange dynamical
systems.

\subsection{Case (1) Representations of $\protect\widetilde{\mathbb{T}}%
\mathbb{=T}_{\infty}/\sim$ and inflation/deflation dynamics}

Define 
\begin{equation*}
\widetilde{\Sigma_{\infty}^{\dag}}=\Sigma_{\infty}^{\dag}/\sim
\end{equation*}
where $\theta\sim\psi$ when there are $p,q\in\mathbb{N}$ such that $\xi
(\theta|p)=\xi(\psi|p)$ and $S^{p}\theta=S^{q}\psi$. Denote the equivalence
class to which $\theta$ belongs by $\left[ \theta\right] $. **We also use
square brackets in another way elsewhere in the paper. We endow $\widetilde{%
\Sigma_{\infty}^{\dag}}$ with the discrete topology for now.

\begin{lemma}
\label{prevlem}A homeomorphism $\widetilde{\Pi}:\widetilde{\Sigma_{\infty
}^{\dag}}\rightarrow\widetilde{\mathbb{T}}$ is well defined by $\widetilde{%
\Pi }(\left[ \theta\right] )=$ $\left[ \Pi(\theta)\right] $.
\end{lemma}

\begin{proof}
This follows from Theorem \ref{main}/10.
\end{proof}

Since the elements $\{s^{a_{1}},s^{a_{2}},...,s^{a_{n}}\}$ are relatively
prime, there is an $M$ such that, for any $m\geq M$, there is an $l$ and
indices $i_{1},i_{2},...,i_{l}$ such that $%
m=a_{i_{1}}+a_{i_{2}}+...+a_{i_{l}}$. Therefore, for given $%
\theta\in\Sigma_{\infty}^{\dag}$ there is a $j$ such that $\xi(\theta|j)>M$,
and there exists $l$ and indices $i_{1},i_{2},...,i_{l}$ such that $%
\xi(i_{1}i_{2}...i_{l})=\xi(\theta|j)-1$. We define a shift map $\widetilde{S%
}:\widetilde{\Sigma_{\infty}^{\dag}}\rightarrow \widetilde{%
\Sigma_{\infty}^{\dag}}$ according to%
\begin{equation*}
\widetilde{S}(\left[ \theta\right] )=\left[ i_{1}i_{2}...i_{l}\theta
_{j}\theta_{j+1}...\right]
\end{equation*}
Likewise, we choose indices $i_{1}^{\prime},i_{2}^{\prime},...,i_{l^{%
\prime}}^{\prime}$ such that $\xi(i_{1}^{\prime}i_{2}^{\prime}...i_{l}^{%
\prime})=\xi(\theta|j)+1$ and define the inverse shift map $\widetilde{S}%
^{-1}:\widetilde{\Sigma_{\infty}^{\dag}}\rightarrow\widetilde{\Sigma_{\infty
}^{\dag}}$ according to 
\begin{equation*}
\widetilde{S}^{-1}(\left[ \theta\right] )=\left[ i_{1}^{\prime}i_{2}^{%
\prime}...i_{l^{\prime}}^{\prime}\theta_{j}\theta_{j+1}...\right]
\end{equation*}
As an example, for the case where $a_{1}=1$ we can choose 
\begin{align*}
\widetilde{S}(\widetilde{\theta}) & =\left[ 11...1\theta_{2}\theta _{3}...%
\right] \text{ where there are initially }\theta_{1}-1\text{ ones,} \\
\widetilde{S}^{-1}(\widetilde{\theta}) & =\left[ 11...1\theta_{2}\theta
_{3}...\right] \text{ where there are initially }\theta_{1}+1\text{ ones}
\end{align*}

\begin{theorem}
\label{conj1}If $(\mathcal{F},\mathcal{G)}$ is locally rigid, then the
symbolic (shift) dynamical system $\widetilde{S}:\widetilde{\Sigma^{\dag}}%
\rightarrow\widetilde{\Sigma^{\dag}}$ is well defined and conjugate to the
deflation/inflation dynamical system $\widetilde{\alpha}:\widetilde{\mathbb{T%
}}\rightarrow\widetilde{\mathbb{T}}$ that is well defined by 
\begin{align*}
\widetilde{\alpha}\left[ \Pi(\theta)\right] & =\left[ \Pi(\widetilde{S}(%
\left[ \theta\right] ))\right] \\
\widetilde{\alpha}^{-1}\left[ \Pi(\theta)\right] & =\left[ \Pi (\widetilde{S}%
^{-1}(\left[ \theta\right] ))\right]
\end{align*}
The following diagrams commute:%
\begin{equation*}
\begin{array}{ccc}
\widetilde{\Sigma^{\dag}} & \overset{\widetilde{S}}{\rightarrow} & 
\widetilde{\Sigma^{\dag}} \\ 
\widetilde{\Pi}\downarrow &  & \downarrow\widetilde{\Pi} \\ 
\widetilde{\mathbb{T}} & \underset{\widetilde{\alpha}}{\rightarrow} & 
\widetilde{\mathbb{T}}%
\end{array}
\text{ and }%
\begin{array}{ccc}
\widetilde{\Sigma^{\dag}} & \overset{\widetilde{S}^{-1}}{\longleftarrow} & 
\widetilde{\Sigma^{\dag}} \\ 
\widetilde{\Pi}^{-1}\uparrow &  & \uparrow\widetilde{\Pi}^{-1} \\ 
\widetilde{\mathbb{T}} & \underset{\widetilde{\alpha}^{-1}}{\longleftarrow}
& \widetilde{\mathbb{T}}%
\end{array}
\text{.}
\end{equation*}
\end{theorem}

\begin{proof}
Follows from Theorem 10 and calculations which prove that the equivalence
classes are respected.
\end{proof}

Theorem \ref{conj1} provides conjugacies between different tilings and their
inflation dynamics. For example, the tiling space associated with the one
dimensional Fibonacci TIFS $\{\mathbb{R}%
^{2};f_{1}(x)=ax,f_{2}=a^{2}x+1-a^{2}\}$ where $\mathcal{T}$ is
two-dimensional translations, is homeomorphic to the Golden-b tiling space
where $\mathcal{T}$ is the two-dimensional euclidean group with reflections.
The shift $S\ $acts conjugately on both systems and results such as both
having the same topological entropy, partition function etc, with respect to
the discrete topology. Another nice family of examples can be constructed
using chair tilings (which are locally rigid with respect to the appropriate
IFS\ and group $\mathcal{T}$).

To conclude this section we examine the relationship between $\widetilde{%
\alpha}$ and $\alpha$ (see Section \ref{canonical}). We make the following
observations, which are based specific results earlier in this paper. The
following observations connect the action of $\alpha^{\xi (\theta|k)}$ on $%
\Pi(\theta),$ the usual shift $S:\Sigma_{\infty}^{\dag
}\rightarrow\Sigma_{\infty}^{\dag}$, and the action of $\widetilde{\alpha }%
^{\xi(\theta|k)}$ on $\left[ \Pi(\theta)\right] $.

\begin{proposition}
If $(\mathcal{F},\mathcal{G)}$ is locally rigid, then for all $\theta\in
\Sigma_{\infty}^{\dag}$, $n\in\left[ N\right] ,$ $k\in\mathbb{N}_{0},$ $%
\widetilde{\alpha}^{\xi(\theta|k)}\left[ \Pi(\theta)\right] =\left[
\alpha^{\xi(\theta|k)}\Pi(\theta)\right] =\left[ \Pi(S^{k}\theta)\right] $
and $\widetilde{\alpha}^{-a_{n}}\left[ \Pi(\theta)\right] =\Pi(n\theta)$
\end{proposition}

\begin{proof}
Follows from Theorem \ref{keythm2} (ii).
\end{proof}

\subsection{Case (2) Representations of $\protect\widehat{\mathbb{T}}=(%
\mathbb{T}_{\infty}\mathbb{\times}\mathcal{T)}/\sim$ and inflation/deflation
dynamics}

In this case the tiling space is 
\begin{equation*}
\widehat{\mathbb{T}}=(\mathbb{T}_{\infty}\mathbb{\times}\mathcal{T)}/\sim
\end{equation*}
where $\mathbb{T}_{\infty}\mathbb{\times}\mathcal{T}$ is equipped with the
metric $d_{\mathbb{T}_{\infty}}+d_{\mathcal{T}}$ and $\Pi(\theta)\times
E\sim\Pi(\psi)\times E^{\prime}$ iff $E\Pi(\theta)=E^{\prime}\Pi(\psi)$ with
the induced metric. The induced metric on $\widehat{\mathbb{T}}$ is denoted $%
d_{\widehat{\mathbb{T}}}$. Here we let $\mathcal{T}$ be any group of
isometries on $\mathbb{R}^{M}$ that contains the group generated by the set
of isometries that map the set of prototiles into the tilings. We assume
that $(\mathcal{F},\mathcal{G)}$ is locally rigid so that Theorem \ref{mainG}
applies. To simplify notation, let the equivalence class in $\widehat{%
\mathbb{T}}$ that contains $\Pi(\theta)\times E$ be 
\begin{equation*}
\widehat{\Psi}(\Pi(\theta),E)=\left\{ \left( \Pi(\psi),E^{\prime}\right)
\in\Sigma_{\infty}^{\dag}\times\mathcal{T}:p,q\in\mathbb{N}_{0},EE^{^{\prime
}-1}=f_{-\psi|q}f_{-\theta|p}^{-1},\xi(\theta|p)=\xi(\psi|q),S^{p}\theta
=S^{q}\psi\right\}
\end{equation*}

Similarly we define, for each $\theta\times E\in\Sigma_{\infty}^{\dag}\times%
\mathcal{T}$,%
\begin{equation*}
\Psi(\theta,E)=\left\{ \left( \psi,E^{\prime}\right) \in\Sigma_{\infty
}^{\dag}\times\mathcal{T}:p,q\in\mathbb{N}_{0},EE^{^{\prime}-1}=f_{-\psi
|q}f_{-\theta|p}^{-1},\xi(\theta|p)=\xi(\psi|q),S^{p}\theta=S^{q}\psi\right%
\} .
\end{equation*}
That is, $\Psi(\theta,E)$ is a member of $\Sigma_{\infty}^{\dag}\times%
\mathcal{T}/\sim$ where $(\theta,E)\sim\left( \psi,E^{\prime}\right) $ iff $%
\Psi(\theta,E)=\Psi\left( \psi,E^{\prime}\right) .$

\begin{lemma}
$\widehat{\Psi}(\Pi(\theta),E)=\widehat{\Psi}(\Pi(\psi),E^{\prime})$ if and
only if $\Psi(\theta,E)=\Psi(\psi,E^{\prime})$.
\end{lemma}

Define a metric space $(\mathcal{X},d_{\mathcal{X}}),$ in the obvious way,
by $\mathcal{X=\{}\Psi(\theta,E):(\theta,E)\in\Sigma_{\infty}^{\dag}\times%
\mathcal{T\}}$ where 
\begin{equation*}
d_{\mathcal{X}}(\Psi(\theta,E),\Psi(\psi,H))=\inf\{d(\theta^{\prime},\psi^{%
\prime})+d_{\mathcal{T}}(E^{\prime},H^{\prime}):(\theta^{\prime
},E^{\prime})\in\Psi(\theta,E),(\psi^{\prime},H^{\prime})\in\Psi(\psi,H)\}
\end{equation*}
is the induced metric.

\begin{lemma}
A homeomorphism $\widehat{\Pi}:\mathcal{X\rightarrow}\widehat{\mathbb{T}}$
is well defined by%
\begin{equation*}
\widehat{\Pi}(\Psi(\theta,E))=\widehat{\Psi}(\Pi\left( \theta\right) ,E)
\end{equation*}
\end{lemma}

Now look at the action of the dynamical systems 
\begin{equation*}
\widehat{S}:\mathcal{X\rightarrow X}\text{ and }\widehat{\alpha}:\widehat{%
\mathbb{T}}\rightarrow\widehat{\mathbb{T}}
\end{equation*}
defined by%
\begin{align*}
\widehat{S}\Psi(\theta,E) & =\widehat{S}\Psi(\theta^{1},Ef_{1}^{-1})=\Psi(S%
\theta^{1},Esf_{1}^{-1}) \\
\widehat{\alpha}\widehat{\Psi}(\Pi(\theta),E) & =\widehat{\alpha }\widehat{%
\Psi}(\Pi(\theta^{1}),Ef_{1}^{-1})=\widehat{\Psi}(\Pi(S\theta
^{1}),Esf_{1}^{-1}).
\end{align*}

\begin{theorem}
Let $(\mathcal{F},\mathcal{G)}$ be locally rigid. Then all of the referenced
transformations in the following diagram are well defined homeomorphisms and
the diagram commutes 
\begin{equation*}
\begin{array}{ccc}
\mathcal{X} & \overset{\widehat{S}}{\rightarrow} & \mathcal{X} \\ 
\widehat{\Pi}\downarrow &  & \downarrow\widehat{\Pi} \\ 
\widehat{\mathbb{T}} & \underset{\widehat{\alpha}}{\rightarrow} & \widehat{%
\mathbb{T}}%
\end{array}
\text{ .}
\end{equation*}
\end{theorem}

**Discuss relationship to \cite{anderson}.

**This desciption implies that $\widehat{\mathbb{T}}$ is an indecomposable
continuum in some standard cases.

\begin{acknowledgement}
We thank Alexandra Grant for careful readings and corrections of various
versions of this work , and for many interesting and helpful discussions and
observations.
\end{acknowledgement}

\end{document}